\documentclass[11pt]{article}
\usepackage{eurosym}
\usepackage{amsmath}
\usepackage{amssymb}
\usepackage{amscd}
\usepackage{xspace}
\usepackage{verbatim}

\providecommand{\U}[1]{\protect\rule{.1in}{.1in}}
\newtheorem{theorem}{Theorem}[section]

\newtheorem{corollary}[theorem]{Corollary}

\newtheorem{definition}[theorem]{Definition}

\newtheorem{lemma}[theorem]{Lemma}

\newtheorem{proposition}[theorem]{Proposition}
\newtheorem{remark}[theorem]{Remark}

\newenvironment{proof}[1][Proof]{\textbf{#1.} }{\hfill\rule{0.5em}{0.5em}}
{\catcode`\@=11\global\let\AddToReset=\@addtoreset
\AddToReset{equation}{section}

\AddToReset{theorem}{section}

\begin{document}

\title{Pointwise estimates and existence of solutions of porous medium and $%
p $-Laplace evolution equations with absorption and measure data}
\author{Marie-Fran\c{c}oise Bidaut-V\'{e}ron\thanks{%
Laboratoire de Math\'{e}matiques et Physique Th\'{e}orique, Facult\'{e} des
Sciences, Universit\'{e} Fran\c{c}ois Rabelais, Tours, France. E-mail:
veronmf@univ-tours.fr} \and Quoc-Hung Nguyen\thanks{%
Laboratoire de Math\'{e}matiques et Physique Th\'{e}orique, Facult\'{e} des
Sciences, Universit\'{e} Fran\c{c}ois Rabelais, Tours, France. E-mail:
Hung.Nguyen-Quoc@lmpt.univ-tours.fr}}
\date{}
\maketitle

\begin{abstract}
Let $\Omega $ be a bounded domain of $\mathbb{R}^{N}(N\geq 2)$. We obtain a
necessary and a sufficient condition, expressed in terms of capacities, for
existence of a solution to the porous medium equation with absorption 
\begin{equation*}
\left\{ 
\begin{array}{l}
{u_{t}}-{\Delta }(|u|^{m-1}u)+|u|^{q-1}u=\mu ~~\text{in }\Omega \times (0,T),
\\ 
{u}=0~~~\text{on }\partial \Omega \times (0,T), \\ 
u(0)=\sigma ,%
\end{array}%
\right.
\end{equation*}%
where $\sigma $ and $\mu $ are bounded Radon measures, $q>\max (m,1)$, $m>%
\frac{N-2}{N}$. We also obtain a sufficient condition for existence of a
solution to the $p$-Laplace evolution equation 
\begin{equation*}
\left\{ 
\begin{array}{l}
{u_{t}}-{\Delta _{p}}u+|u|^{q-1}u=\mu ~~\text{in }\Omega \times (0,T), \\ 
{u}=0~~~\text{on }\partial \Omega \times (0,T), \\ 
u(0)=\sigma .%
\end{array}%
\right.
\end{equation*}%
where $q>p-1$ and $p>2$. %\pagebreak 
\end{abstract}

\tableofcontents

\medskip

\noindent{\small \textit{keywords: Sobolev-Besov capacities; Bessel
capacities; Radon measures; renormalized solutions.}\newline
 }

\section{\protect\small \ Introduction and main results}

{\small Let $\Omega $ be a bounded domain of $\mathbb{R}^{N}$, $N\geq 2$ and 
$T>0,$ and $\Omega _{T}=\Omega \times (0,T)$. In this paper we study the
existence of solutions to the following two types of evolution problems: the
porous medium problem with absorption 
\begin{equation}
\left\{ 
\begin{array}{l}
{u_{t}}-{\Delta }(|u|^{m-1}u)+|u|^{q-1}u=\mu ~~\text{in }\Omega _{T}, \\ 
{u}=0~~~\text{on }\partial \Omega \times (0,T), \\ 
u(0)=\sigma ,%
\end{array}%
\right.  \label{one}
\end{equation}%
where $m>\frac{N-2}{N}$ and $q>\max (1,m),$ and the $p$-Laplace evolution
problem with absorption 
\begin{equation}
\left\{ 
\begin{array}{l}
{u_{t}}-{\Delta _{p}}u+|u|^{q-1}u=\mu ~~\text{in }\Omega _{T}, \\ 
{u}=0~~~\text{on }\partial \Omega \times (0,T), \\ 
u(0)=\sigma ,%
\end{array}%
\right.  \label{two}
\end{equation}%
where $q>p-1>1,$ and $\mu $ and $\sigma $ are bounded Radon measures
respectively on $\Omega _{T}$ and $\Omega $. In the sequel, for any bounded
domain $O$ of $\mathbb{R}^{l}(l\geq 1)$, we denote by $\mathcal{M}_{b}(O)$
the set of bounded Radon measures in $O$, and by $\mathcal{M}_{b}^{+}(O)$
its positive cone. For any $\nu \in \mathcal{M}_{b}(O),$ we denote by $\nu
^{+}$ and $\nu ^{-}$ respectively its positive and negative part. \medskip }

{\small When $m=1,p=2$ and $q>1$ the problem has been studied by Brezis and
Friedman \cite{BrFr} with $\mu =0.$ It is shown that in the subcritical case 
$q<1+2/N$, the problem can be solved for any $\sigma \in \mathcal{M}%
_{b}(\Omega ),$ and it has no solution when $q\geq 1+2/N$ and $\sigma $ is a
Dirac mass. The general case has been solved by Baras and Pierre \cite{BaPi1}
and their results are expressed in terms of capacities. For $s>1,\alpha >0$,
the capacity $\text{Cap}_{\mathbf{G}_{\alpha },s}$ of a Borel set $E\subset 
\mathbb{R}^{N}$, defined by 
\begin{equation*}
\text{Cap}_{\mathbf{G}_{\alpha },s}(E)=\inf \{||g||_{L^{s}(\mathbb{R}%
^{N})}^{s}:g\in L_{+}^{s}(\mathbb{R}^{N}),\mathbf{G}_{\alpha }\ast g\geq 1%
\text{ on }E\},
\end{equation*}%
where $\mathbf{G}_{\alpha }$ is the Bessel kernel of order $\alpha $ and the
capacity $\text{Cap}_{2,1,s}$ of a compact set $K\subset \mathbb{R}^{N+1}$
is defined by 
\begin{equation*}
\text{Cap}_{2,1,s}(K)=\inf \left\{||\varphi ||_{W_{s}^{2,1}(\mathbb{R}%
^{N+1})}^{s}:\varphi \in S(\mathbb{R}^{N+1}),\varphi \geq 1\text{ in a
neighborhood of}~K\right\},
\end{equation*}%
where 
\begin{equation*}
||\varphi ||_{W_{s}^{2,1}(\mathbb{R}^{N+1})}=||\varphi ||_{L^{s}(\mathbb{R}%
^{N+1})}+||\varphi _{t}||_{L^{s}(\mathbb{R}^{N+1})}+||\left\vert \nabla
\varphi \right\vert ||_{L^{s}(\mathbb{R}^{N+1})}+\sum%
\limits_{i,j=1,2,...,N}||\varphi _{x_{i}x_{j}}||_{L^{s}(\mathbb{R}^{N+1})}.
\end{equation*}%
The capacity $\text{Cap}_{2,1,s}$ is extended to Borel sets by the usual
method. Note the relation between the two capacities: 
\begin{equation*}
C^{-1}\text{Cap}_{\mathbf{G}_{2-\frac{2}{s}},s}(E)\leq \text{Cap}%
_{2,1,s}(E\times \{0\})\leq C\text{Cap}_{\mathbf{G}_{2-\frac{2}{s}},s}(E)
\end{equation*}%
for any Borel set $E\subset \mathbb{R}^{N}$, see \cite[Corollary 4.21]{H1}.
In particular, for any $\omega \in \mathcal{M}_{b}(\mathbb{R}^{N})$ and $%
a\in \mathbb{R}$, the measure $\omega \otimes \delta _{\{t=a\}}$ in $\mathbb{%
R}^{N+1}$ is absolutely continuous with respect to the capacity $\text{Cap}%
_{2,1,s}$ ( in $\mathbb{R}^{N+1}$) if and only if $\omega $ is absolutely
continuous with respect to the capacity $\text{Cap}_{\mathbf{G}_{2-\frac{2}{s%
}},s}$ (in $\mathbb{R}^{N}$). \newline
From \cite{BaPi1}, the problem 
\begin{equation*}
\left\{ 
\begin{array}{l}
{u_{t}}-{\Delta }u+|u|^{q-1}u=\mu ~~\text{in }\Omega _{T}, \\ 
{u}=0~~~\text{on }\partial \Omega \times (0,T), \\ 
u(0)=\sigma ,%
\end{array}%
\right.
\end{equation*}%
has a solution if and only if the measures $\mu $ and $\sigma $ are
absolutely continuous with respect to the capacities $\text{Cap}%
_{2,1,q^{\prime }}$ in $\Omega _{T}\text{ and Cap}_{\mathbf{G}_{\frac{2}{q}%
},q^{\prime }}$ in $\Omega $ respectively, where $q^{\prime }=\frac{q}{q-1}$%
.\medskip \medskip }

{\small In Section \ref{PM} we study problem (\ref{one}). \medskip }

{\small For $m>1$, Chasseigne \cite{Ch} has extended the results of \cite%
{BrFr} for $\mu =0$ in the new subcritical range $m<q<m+\frac{2}{N}$. The
supercritical case $q\geq m+\frac{2}{N}$ with $\mu =0$ and $\sigma $ is
positive is studied in \cite{Ch0}. He has essentially proved that if problem %
\eqref{one} has a solution, then $\sigma \otimes \delta _{\{t=0\}}$ is
absolutely continuous with respect to the capacity $\text{Cap}_{2,1,\frac{q}{%
q-m},q^{\prime }}$, defined for anycompact set $K\subset \mathbb{R}^{N+1}$
by 
\begin{equation*}
\text{Cap}_{2,1,\frac{q}{q-m},q^{\prime }}(K)=\inf \left\{||\varphi ||_{W_{%
\frac{q}{q-m},q^{\prime }}^{2,1}(\mathbb{R}^{N+1})}^{\frac{q}{q-m}}:\varphi
\in S(\mathbb{R}^{N}),\varphi \geq 1\text{ in a neighborhood of}~E\right\},
\end{equation*}%
where 
\begin{equation*}
||\varphi ||_{W_{\frac{q}{q-m},q^{\prime }}^{2,1}(\mathbb{R}%
^{N+1})}=||\varphi ||_{L^{\frac{q}{q-m}}(\mathbb{R}^{N+1})}+||\varphi
_{t}||_{L^{q^{\prime }}(\mathbb{R}^{N+1})}+||\left\vert \nabla \varphi
\right\vert ||_{L^{\frac{q}{q-m}}(\mathbb{R}^{N+1})}+\sum%
\limits_{i,j=1,2,...,N}||\varphi _{x_{i}x_{j}}||_{L^{\frac{q}{q-m}}(\mathbb{R%
}^{N+1})}.
\end{equation*}%
\bigskip }

{\small In this Section, we first give \textit{necessary conditions} on the
measures $\mu $ and $\sigma $ for existence, which cover the results
mentioned above. }

\begin{theorem}
{\small \label{NCE}Let $q>\max (1,m)$ and $\mu \in \mathcal{M}_{b}(\Omega
_{T})$ and $\sigma \in \mathcal{M}_{b}(\Omega )$. If problem (\ref{one}) has
a very weak solution then $\mu$ and $\sigma \otimes \delta _{\{t=0\}}$ are
absolutely continuous with respect to the capacity $\text{Cap}_{2,1,\frac{q}{%
q-m},\frac{q}{q-1}}$. }
\end{theorem}

\begin{remark}
{\small It is easy to see that the capacity $\text{Cap}_{2,1,\frac{q}{q-m},%
\frac{q}{q-1}}$ is absolutely continuous with respect to the capacity $\text{%
Cap}_{2,1,\frac{q}{q-\max \{m,1\}}}$. Therefore $\mu$ and $\sigma \otimes
\delta _{\{t=0\}}$ are absolutely continuous with respect to the capacities $%
\text{Cap}_{2,1,\frac{q}{q-\max \{m,1\}}}$.In particular $\sigma $ is
absolutely continuous with respect to the capacity $\text{Cap}_{\mathbf{G}_{%
\frac{2\max \{m,1\}}{q}},\frac{q}{q-\max \{m,1\}}}$. }
\end{remark}

{\small The main result of this Section is the following \textit{sufficient
condition} for existence, where we use the notion of $R$-truncated Riesz
parabolic potential $\mathbb{I}_{2}$ on $\mathbb{R}^{N+1}$ of a measure $\mu
\in \mathcal{M}_{b}^{+}(\Omega _{T})$ , defined by 
\begin{equation*}
\mathbb{I}_{2}^{R}[\mu ](x,t)=\int_{0}^{R}\frac{\mu (\tilde{Q}_{\rho }(x,t))%
}{\rho ^{N}}\frac{d\rho }{\rho } ~~\text{ for any }~(x,t)\in \mathbb{R}%
^{N+1},
\end{equation*}%
with $R\in (0,\infty ]$, and $\tilde{Q}_{\rho }(x,t)=B_{\rho }(x)\times
(t-\rho ^{2},t+\rho ^{2})$. }

\begin{theorem}
{\small \label{1106201412}Let $m>\frac{N-2}{N}$, $q>\max (1,m)$, $\mu \in 
\mathcal{M}_{b}(\Omega _{T})$ and $\sigma \in \mathcal{M}_{b}(\Omega )$. }

\begin{description}
\item[i.] {\small If $m>1$ and $\mu$ and $\sigma $ are absolutely continuous
with respect to the capacities $\text{Cap}_{2,1,q^{\prime }}$ in $\Omega
_{T} $ and $\text{Cap}_{\mathbf{G}_{\frac{2}{q}},q^{\prime }}$ in $\Omega ,$
then there exists a very weak solution $u$ of (\ref{one}), satisfying for $%
a.e.(x,t)\in \Omega _{T}$ 
\begin{equation}  \label{1106201430}
|u(x,t)|\leq C\left( \left( \frac{|\sigma |(\Omega )+|\mu |(\Omega _{T})}{%
d^{N}}\right) ^{m_{1}}+|\sigma |(\Omega )+|\mu |(\Omega _{T})+1+\mathbb{I}%
_{2}^{2d}[|\sigma |\otimes \delta _{\{t=0\}}+|\mu |](x,t)\right) ,
\end{equation}%
where $C=C(N,m)>0$ and 
\begin{equation*}
m_{1}=\frac{(N+2)(2mN+1)}{m(mN+2)(1+2N)},\qquad d=\text{diam}(\Omega
)+T^{1/2}.
\end{equation*}
}

\item[ii.] {\small If $\frac{N-2}{N}<m\leq 1,$ and $\mu$ and $\sigma $ are
absolutely continuous with respect to the capacities $\text{Cap}_{2,1,\frac{%
2q}{2(q-1)+N(1-m)}}$ in $\Omega _{T}$ and $\text{Cap}_{\mathbf{G}_{\frac{%
2-N(1-m)}{q}},\frac{2q}{2(q-1)+N(1-m)}}$ in $\Omega ,$ there exists a very
weak solution $u$ $of$ (\ref{one}), such that for $a.e.(x,t)\in \Omega _{T}$ 
\begin{equation}  \label{1106201430b}
|u(x,t)|\leq C\left( \left( \frac{|\sigma |(\Omega )+|\mu |(\Omega _{T})}{%
d^{N}}\right) ^{m_{2}}+1+\left( \mathbb{I}_{2}^{2d}[|\sigma |\otimes \delta
_{\{t=0\}}+|\mu |](x,t)\right) ^{\frac{2}{2-N(1-m)}}\right) ,
\end{equation}%
where $C=C(N,m)>0$ and 
\begin{equation*}
m_{2}=\frac{2N(N+2)(m+1)}{(2+Nm)(2-N(1-m))(2+N(1+m))}.
\end{equation*}%
. }
\end{description}
\end{theorem}

\begin{remark}
{\small These estimates are not homogeneous in $u.$ In particular if $\mu
\equiv 0,$ $u$ satisfies the decay estimates, for $a.e.$ $(x,t)\in \Omega
_{T},$ }

\begin{description}
\item[i.] {\small if $m>1,$ 
\begin{equation*}
|u(x,t)|\leq C\left( \left(\frac{|\sigma|(\Omega )}{d^{N}}%
\right)^{m_1}+|\sigma|(\Omega )+1+\frac{|\sigma|(\Omega )}{Nt^{N/2}}\right) ,
\end{equation*}
}

\item[ii.] {\small if $m<1,$ 
\begin{equation*}
| u(x,t)|\leq C\left( \left( \frac{|\sigma|(\Omega )}{ d^{N}}\right)
^{m_{2}}+1+\left( \frac{|\sigma|(\Omega )}{Nt^{N/2}} \right) ^{\frac{2}{%
2-N(m-1)}}\right) .
\end{equation*}
}
\end{description}
\end{remark}

{\small We also give other types of \textit{sufficien}t conditions for 
\textit{measures which are good in time}, that means such that 
\begin{equation}
\sigma \in L^{1}(\Omega )\text{ \quad and }|\mu |\leq f+\omega \otimes F,%
\text{ \quad where }f\in L_{+}^{1}(\Omega _{T}),F\in L^{1}_+((0,T)),
\label{gm}
\end{equation}%
see Theorem \ref{1106201432}. The proof is based on estimates for the
stationary problem in terms of elliptic Riesz potential.\bigskip }

{\small In Section \ref{PL}, we consider problem (\ref{two}). Let us recall
some former results about it\medskip . }

{\small For $q>p-1>0,$ Pettitta, Ponce and Porretta \cite{PePoPor} have
proved that it admits a (unique renormalized) solution provided $\sigma \in
L^{1}(\Omega )$ and $\mu \in \mathcal{M}_{b}(\Omega _{T})$ is a \textit{%
diffuse} measure, i.e. absolutely continuous with respect to $C_{p}$%
-capacity in $\Omega _{T}$, defined on a compact set $K\subset \Omega _{T}$
by 
\begin{equation}
C_{p}(K,\Omega _{T})=\inf \left\{ ||\varphi ||_{W}:\varphi \in C_{c}^{\infty
}(\Omega _{T})\varphi \geq 1\text{ on }K\right\} ,  \label{aaa}
\end{equation}%
where 
\begin{equation*}
W=\{z:z\in L^{p}(0,T,W_{0}^{1,p}(\Omega )\cap L^{2}(\Omega )),z_{t}\in
L^{p^{\prime }}(0,T,W^{-1,p^{\prime }}(\Omega )+L^{2}(\Omega ))\}.
\end{equation*}
}

{\small In the recent work \cite{BiH}, we have proved a stability result for
the $p$-Laplace parabolic equation, see Theorem \ref{100620145}, for $p>%
\frac{2N+1}{N+1}$. As a first consequence, in the new subcritical range 
\begin{equation*}
q<p-1+\frac{p}{N},
\end{equation*}%
problem (\ref{two}) admits a renormalized solution for any measures $\mu \in 
\mathcal{M}_{b}(\Omega _{T})$ and $\sigma \in L^1(\Omega ).$ Moreover, we
have obtained sufficient conditions for existence, for measures that have a 
\textit{good behavior in time}, of the form (\ref{gm}). It is shown that (%
\ref{two}) has a renormalized solution if $\omega \in \mathcal{M}%
_{b}^{+}(\Omega )$ is absolutely continuous with respect to $\text{Cap}_{%
\mathbf{G}_{p},\frac{q}{q-p+1}}$. The proof is based on estimates of \cite%
{BiHV} for the stationary problem which involve Wolff potentials. \medskip
\bigskip }

{\small Here we give \textit{new sufficient conditions when }$p>2.$ The next
Theorem is our second main result: }

\begin{theorem}
{\small \label{100620149}Let $q>p-1>1$ and $\mu \in \mathcal{M}_{b}(\Omega
_{T})$ and $\sigma \in \mathcal{M}_{b}(\Omega )$. If $\mu$ and $\sigma $ are
absolutely continuous with respect to the capacities $\text{Cap}%
_{2,1,q^{\prime }}$ in $\Omega _{T}$ and $\text{Cap}_{\mathbf{G}_{\frac{2}{q}%
},q^{\prime }}$ in $\Omega $, then there exists a distribution solution of
problem (\ref{two}) which satisfies the pointwise estimate 
\begin{equation}  \label{080720141}
|u(x,t)|\leq C\left( 1+D+\left( \frac{|\sigma |(\Omega )+|\mu |(\Omega _{T})%
}{D^{N}}\right) ^{m_3}+\mathbb{I}_{2}^{2D}\left[ |\sigma |\otimes \delta
_{\{t=0\}}+|\mu |\right](x,t) \right)
\end{equation}
for a.e $(x,t)\in \Omega_T$ with $C=C(N,p)$ and 
\begin{equation}
m_3=\frac{(N+p)(\lambda +1)(p-1)}{((p-1)N+p)(1+\lambda (p-1))},\qquad
\lambda =\min \{1/(p-1),1/N\},\qquad D=\text{diam}(\Omega )+T^{1/p}.
\label{lim}
\end{equation}%
Moreover, if $\sigma \in L^{1}(\Omega )$, $u$ is a renormalized solution. }
\end{theorem}

\section{{\protect\small Porous medium equation\label{PM}}}

{\small For $k>0$ and $s\in \mathbb{R}$ we set $T_{k}(s)=\max \{\min
\{s,k\},-k\}$. The solutions of (\ref{one}) are considered in a weak sense: }

\begin{definition}
{\small Let $\mu \in \mathcal{M}_{b}(\Omega _{T})$ and $\sigma \in \mathcal{M%
}_{b}(\Omega )$ and $g\in C(\mathbb{R})$. }

{\small \textbf{i.} A function $u$ is a weak solution of problem 
\begin{equation}
\left\{ 
\begin{array}{l}
{u_{t}}-{\Delta }(|u|^{m-1}u)+g(u)=\mu ~~\text{in }\Omega _{T}, \\ 
{u}=0~~~\text{on }\partial \Omega \times (0,T), \\ 
u(0)=\sigma ~~~\text{ in }~\Omega .%
\end{array}%
\right.  \label{G}
\end{equation}%
if $u\in C(\left[ 0,T\right] ;L^{2}(\Omega )),$ $|u|^{m}\in
L^{2}((0,T);H_{0}^{1}(\Omega ))$ and $g(u)\in L^{1}(\Omega _{T}),$ and for
any $\varphi \in C_{c}^{2,1}(\Omega \times \lbrack 0,T)),$ 
\begin{equation*}
-\int_{\Omega _{T}}u\varphi _{t}dxdt+\int_{\Omega _{T}}\nabla (\left\vert
u\right\vert ^{m-1}u).\nabla \varphi dxdt+\int_{\Omega _{T}}g(u)\varphi
dxdt=\int_{\Omega _{T}}\varphi d\mu +\int_{\Omega }\varphi (0)d\sigma .
\end{equation*}
}

{\small \textbf{ii.} A function $u$ is a very weak solution of (\ref{G}) if $%
u\in L^{\max\{m,1\}}(\Omega _{T})$ and $g(u)\in L^{1}(\Omega _{T}),$ and for
any $\varphi \in C_{c}^{2,1}(\Omega \times \lbrack 0,T)),$ 
\begin{equation*}
-\int_{\Omega _{T}}u\varphi _{t}dxdt-\int_{\Omega _{T}}|u|^{m-1}u\Delta
\varphi dxdt+\int_{\Omega _{T}}g(u)\varphi dxdt=\int_{\Omega _{T}}\varphi
d\mu +\int_{\Omega }\varphi (0)d\sigma .
\end{equation*}
}
\end{definition}

{\small First we give a priori estimates for the problem without
perturbation term: }

\begin{proposition}
{\small \label{110620146}Let $u\in L^{\infty }(\Omega _{T})$ with $%
|u|^{m}\in L^{2}((0,T);H_{0}^{1}(\Omega ))$ be a weak solution to problem 
\begin{equation}
\left\{ 
\begin{array}{l}
{u_{t}}-{\Delta }(|u|^{m-1}u)=\mu ~~\text{in }\Omega _{T}, \\ 
{u}=0~~~\text{on }\partial \Omega \times (0,T), \\ 
u(0)=\sigma ~~\text{ in }~\Omega ,%
\end{array}%
\right.  \label{PME}
\end{equation}%
with $\sigma \in C_{b}(\Omega )$ and $\mu \in C_{b}(\Omega _{T})$. Then, 
\begin{align}
& ||u||_{L^{\infty }((0,T);L^{1}(\Omega ))}\leq |\sigma |(\Omega )+|\mu
|(\Omega _{T}),  \label{110620148} \\
& ||u||_{L^{m+2/N,\infty }(\Omega _{T})}\leq C_{1}(|\sigma |(\Omega )+|\mu
|(\Omega _{T}))^{\frac{N+2}{mN+2}},  \label{110620149} \\
& |||\nabla (|u|^{m-1}u)|||_{L^{\frac{mN+2}{mN+1},\infty }(\Omega _{T})}\leq
C_{2}(|\sigma |(\Omega )+|\mu |(\Omega _{T}))^{\frac{m(N+1)+1}{mN+2}},
\label{1106201410}
\end{align}%
where $C_{1}=C_{1}(N,m),C_{2}=C_{2}(N,m)$. }
\end{proposition}

\begin{proof}[Proof of Proposition \protect\ref{110620146}]
{\small For any $\tau \in (0,T),$ and $k>0$ we have 
\begin{equation*}
\int_{\Omega _{\tau }}(H_{k}(u))_{t}dxdt+\int_{\Omega _{\tau }}|\nabla
T_{k}(|u|^{m-1}u)|^{2}dxdt=\int_{\Omega _{\tau }}T_{k}(|u|^{m-1}u)d\mu (x,t),
\end{equation*}%
where $H(a)=\int_{0}^{a}T_{k}(|y|^{m-1}y)dy$. This leads to 
\begin{align}
& \int_{\Omega _{T}}|\nabla T_{k}(|u|^{m-1}u)|^{2}dxdt\leq k(|\sigma
|(\Omega )+|\mu |(\Omega _{T}))~~\text{ and }  \label{110620147} \\
& \int_{\Omega }(H_{k}(u))(\tau )dx\leq k(|\sigma |(\Omega )+|\mu |(\Omega
_{T})),~\forall \tau \in (0,T).  \notag
\end{align}%
Since $H_{k}(a)\geq k(|a|-k)$ for any $a$ and $k>0$, we find 
\begin{equation*}
\int_{\Omega }(|u|(\tau )-k)dx\leq |\sigma |(\Omega )+|\mu |(\Omega
_{T}),~\forall \tau \in (0,T).
\end{equation*}%
Letting $k\rightarrow 0$, we get \eqref{110620148}.\bigskip }

{\small Next we prove \eqref{110620149}. By the Gagliardo-Nirenberg
embedding theorem, there holds 
\begin{align*}
\int_{\Omega _{T}}|T_{k}(|u|^{m-1}u)|^{\frac{2(N+1)}{N}}dxdt& \leq
C_{1}||T_{k}(|u|^{m-1}u)||_{L^{\infty }((0,T);L^{1}(\Omega
))}^{2/N}\int_{\Omega _{T}}|\nabla T_{k}(|u|^{m-1}u)|^{2}dxdt \\
& \leq C_{1}k^{\frac{2(m-1)}{mN}}||u||_{L^{\infty }((0,T);L^{1}(\Omega
))}^{2/N}\int_{\Omega _{T}}|\nabla T_{k}(|u|^{m-1}u)|^{2}dxdt.
\end{align*}%
Thus, from \eqref{110620147} and \eqref{110620148} we get 
\begin{equation*}
k^{\frac{2(N+1)}{N}}|\{|u|^{m}>k\}|\leq \int_{\Omega
_{T}}|T_{k}(|u|^{m-1}u)|^{\frac{2(N+1)}{N}}dxdt\leq c_{1}k^{\frac{2(m-1)}{mN}%
+1}(|\sigma |(\Omega )+|\mu |(\Omega _{T}))^{\frac{N+2}{N}},
\end{equation*}%
which implies \eqref{110620149}. Finally, we prove \eqref{1106201410}.
Thanks to \eqref{110620147} and \eqref{110620149} we have for $k,k_{0}>0$ 
\begin{align*}
|\{|\nabla (|u|^{m-1}u)|>k\}|& \leq \frac{1}{k^{2}}\int_{0}^{k^{2}}|\{|%
\nabla (|u|^{m-1}u)|>\ell \}|d\ell \\
& \leq |\{|u|^{m}>k_{0}\}|+\frac{1}{k^{2}}\int_{\Omega _{T}}|\nabla
T_{k_{0}}(|u|^{m-1}u)|^{2}dxdt \\
& \leq C_{1}k_{0}^{-\frac{2}{mN}-1}(|\sigma |(\Omega )+|\mu |(\Omega _{T}))^{%
\frac{N+2}{N}}+k_{0}k^{-2}(|\sigma |(\Omega )+|\mu |(\Omega _{T})).
\end{align*}%
Choosing $k_{0}=k^{\frac{Nm}{Nm+1}}(|\sigma |(\Omega )+|\mu |(\Omega _{T}))^{%
\frac{m}{Nm+1}},$ we get \eqref{1106201410}. }
\end{proof}\medskip

{\small Next we show the necessary conditions given at Theorem \ref{NCE}. 
\newline
}

\begin{proof}[Proof of Theorem \protect\ref{NCE}]
{\small As in \cite[Proof of Proposition 3.1]{BaPi1}, it is enough to claim
that for any compact $K\subset \Omega \times \lbrack 0,T)$ such that $\mu
^{-}(K)=0$, $(\sigma ^{-}\otimes \delta _{\{t=0\}})(K)=0$ and $\text{Cap}%
_{2,1,\frac{q}{q-m},q^{\prime }}(K)=0$ then $\mu ^{+}(K)=0$ and $(\sigma
^{+}\otimes \delta _{\{t=0\}})(K)=0$. Let $\varepsilon >0$ and choose an
open set $O$ such that $(|\mu |+|\sigma |\otimes \delta
_{\{t=0\}})(O\backslash K)<\varepsilon $ and $K\subset O\subset \Omega
\times (-T,T)$. One can find a sequence $\{\varphi _{n}\}\subset
C_{c}^{\infty }(O)$ which satisfies $0\leq \varphi _{n}\leq 1$, $\left.
\varphi _{n}\right\vert _{K}=1$ and $\varphi _{n}\rightarrow 0$ in $W_{\frac{%
q}{q-m},q^{\prime }}^{2,1}(\mathbb{R}^{N+1})$ and almost everywhere in $O$
(see \cite[Proposition 2.2]{BaPi1}). We get 
\begin{align*}
\int_{\Omega _{T}}\varphi _{n}d\mu +\int_{\Omega }\varphi _{n}(0)d\sigma &
=-\int_{\Omega _{T}}u(\varphi _{n})_{t}dxdt-\int_{\Omega
_{T}}|u|^{m-1}u\Delta \varphi _{n}dxdt+\int_{\Omega _{T}}|u|^{q-1}u\varphi
_{n}dxdt \\
& \leq (||u||_{L^{q}(\Omega _{T})}+||u||_{L^{q}(\Omega _{T})}^{m})||\varphi
_{n}||_{W_{\frac{q}{q-m},\frac{q}{q-1}}^{2,1}(\mathbb{R}^{N+1})}+\int_{%
\Omega _{T}}|u|^{q}\varphi _{n}dxdt.
\end{align*}%
Note that 
\begin{align*}
\int_{\Omega _{T}}\varphi _{n}d\mu +\int_{\Omega }\varphi _{n}(0)d\sigma &
\geq \mu ^{+}(K)+(\sigma ^{+}\otimes \delta _{\{t=0\}})(K)-(|\mu |+|\sigma
|\otimes \delta _{\{t=0\}})(O\backslash K) \\
& \geq \mu ^{+}(K)+(\sigma ^{+}\otimes \delta _{\{t=0\}})(K)-\varepsilon .
\end{align*}%
This implies 
\begin{equation*}
\mu ^{+}(K)+(\sigma ^{+}\otimes \delta _{\{t=0\}})(K)\leq
(||u||_{L^{q}(\Omega _{T})}+||u||_{L^{q}(\Omega _{T})}^{m})||\varphi
_{n}||_{W_{\frac{q}{q-m},\frac{q}{q-1}}^{2,1}(\mathbb{R}^{N+1})}+\int_{%
\Omega _{T}}|u|^{q}\varphi _{n}dxdt+\varepsilon .
\end{equation*}%
Letting the limit we get $\mu ^{+}(K)+(\sigma ^{+}\otimes \delta
_{\{t=0\}})(K)\leq \varepsilon $. Therefore, $\mu ^{+}(K)=(\sigma
^{+}\otimes \delta _{\{t=0\}})(K)=0$.\medskip }
\end{proof}

{\small Next we look for sufficient conditions of existence.}
{\small The crucial result used to establish Theorem \ref%
{1106201412} is the following a priori estimates, due to of Liskevich and
Skrypnik \cite{LSS2} for $m\geq 1$ and Bogelein, Duzaar and Gianazza \cite%
{BoDuGi} for $m\leq 1$. }

\begin{theorem}
{\small \label{1106201413} Let $m>\frac{N-2}{N}$ and $\mu \in (C_{b}(\Omega
_{T}))^{+}$. Let $u\in L^{\infty}_+(\Omega _{T})$ with $u^{m}\in
L^{2}(0,T,H_{loc}^{1}(\Omega ))$ be a weak solution to equation 
\begin{equation*}
{u_{t}}-{\Delta }(u^{m})=\mu ~~\text{in }\Omega _{T}.
\end{equation*}%
Then there exists $C=C(N,m)$ such that, for almost all $(y,\tau )\in \Omega
_{T}$ and any cylinder $\tilde{Q}_{r}(y,\tau )\subset \subset \Omega _{T},$
there holds }

\begin{description}
\item[i.] {\small if $m>1$ 
\begin{equation*}
u(y,\tau )\leq C\left( \left( \frac{1}{r^{N+2}}\int_{\tilde{Q}_{r}(y,\tau
)}|u|^{m+\frac{1}{2N}}dxdt\right) ^{\frac{2N}{1+2N}}+||u||_{L^{\infty
}((\tau -r^{2},\tau +r^{2});L^{1}(B_{r}(y)))}+1+\mathbb{I}_{2}^{2r}[\mu
](y,\tau )\right) ,
\end{equation*}
}

\item[ii.] {\small if $m\leq 1,$ 
\begin{equation*}
u(y,\tau )\leq C\left( \left( \frac{1}{r^{N+2}}\int_{\tilde{Q}_{r}(y,s)}|u|^{%
\frac{2(1+mN)}{N(1+m)}}dxdt\right) ^{\frac{2N(m+1)}{(2-N(1-m))(2+N(1+m))}%
}+1+\left( \mathbb{I}_{2}^{2r}[\mu ](y,\tau )\right) ^{\frac{2}{2-N(1-m)}%
}\right) .
\end{equation*}
}
\end{description}
\end{theorem}

{\small As a consequence we get a new a priori estimate for the porous
medium equation: }

\begin{corollary}
{\small \label{1106201417}Let $m>\frac{N-2}{N}$ and $\mu \in C_{b}(\Omega
_{T})$. Let $u\in L^{\infty }(\Omega _{T})$ with $|u|^{m}\in
L^{2}(0,T,H_{0}^{1}(\Omega ))$ be the weak solution of problem 
\begin{equation*}
\left\{ 
\begin{array}{l}
{u_{t}}-{\Delta }(|u|^{m-1}u)=\mu ~~\text{in }\Omega _{T}, \\ 
{u}=0~~~\text{on }\partial \Omega \times (0,T), \\ 
u(0)=0~~\text{ in }~\Omega .%
\end{array}%
\right.
\end{equation*}%
Then there exists $C=C(N,m)$ such that, for a.e. $(y,\tau )\in \Omega _{T}$, 
}

\begin{description}
\item[i.] {\small if $m>1,$ 
\begin{equation}
|u(y,\tau )|\leq C\left( \left(\frac{|\mu |(\Omega _{T})}{d^{N}}%
\right)^{m_1}+|\mu |(\Omega _{T})+1+\mathbb{I}_{2}^{2d}[|\mu |](y,\tau
)\right),  \label{1106201419}
\end{equation}
}

\item[ii.] {\small if $m\leq 1$, 
\begin{equation}
|u(y,\tau)|\leq C\left( \left( \frac{|\mu |(\Omega _{T})}{d^{N}}\right)
^{m_{2}}+1+\left( \mathbb{I}_{2}^{2d_{1}}[|\mu |](y,\tau)\right) ^{\frac{2}{%
2-N(1-m)}}\right),  \label{1106201419b}
\end{equation}
}
\end{description}

{\small where $m_{1},m_{2}$ and $d$ are defined in Theorem \ref{1106201412}. 
}
\end{corollary}

\begin{proof}
{\small Let $x_{0}\in \Omega ,$ and $Q=B_{2d}(x_{0})\times
(-(2d)^{2},(2d)^{2}).$ Consider the function $U\in (C_b(Q))^+,$ with $%
U^{m}\in L^{p}((-(2d)^{2},(2d)^{2});H_{0}^{1}(B_{2d}(x_{0})))$ such that $U$
is weak solution of 
\begin{equation}
\left\{ 
\begin{array}{l}
{U_{t}}-{\Delta }(U^{m})=\chi _{\Omega _{T}}|\mu |\qquad \text{in }%
B_{2d}(x_{0})\times (-(2d)^{2},(2d)^{2}), \\ 
{U}=0~~~~~~~~~~~~~~~~\text{on }\partial B_{2d}(x_{0})\times
(-(2d)^{2},(2d)^{2}), \\ 
U(-(2d)^{2})=0~~~\text{ in }~B_{2d}(x_{0}).%
\end{array}%
\right.  \label{solU}
\end{equation}%
From Theorem \ref{1106201413}, we get, for a.e $(y,\tau )\in \Omega _{T},$ 
\begin{equation*}
U(y,\tau )\leq c_{1}\left( \left( \frac{1}{d^{N+2}}\int_{\tilde{Q}%
_{d}(y,\tau )}|U|^{m+\frac{1}{2N}}dxdt\right) ^{\frac{2N}{1+2N}%
}+||U||_{L^{\infty }((\tau -d^{2},\tau +d^{2});L^{1}(B_{d}(y)))}+1+\mathbb{I}%
_{2}^{2d}[|\mu |](y,\tau )\right)
\end{equation*}%
if $m>1$ and 
\begin{equation*}
U(y,\tau )\leq C\left( \left( \frac{1}{d^{N+2}}\int_{\tilde{Q}_{d}(y,s)}|u|^{%
\frac{2(1+mN)}{N(1+m)}}dxdt\right) ^{\frac{2N(m+1)}{(2-N(1-m))(2+N(1+m))}%
}+1+\left( \mathbb{I}_{2}^{2r}[\mu ](y,\tau )\right) ^{\frac{2}{2-N(1-m)}%
}\right)
\end{equation*}%
if $m\leq 1$. By Proposition \ref{110620146}, we have 
\begin{align*}
& ||U||_{L^{\infty }((\tau -d^{2},\tau +d^{2});L^{1}(B_{d}(y)))}\leq |\mu
|(\Omega _{T}), \\
& |\{{|U|>\ell }\}|\leq c_{2}(|\mu |(\Omega _{T}))^{\frac{2+N}{N}}\ell ^{-%
\frac{2}{N}-m},\qquad \forall \ell >0.
\end{align*}%
Thus, for any $\ell _{0}>0,$ 
\begin{align*}
& \int_{Q}U^{m+\frac{1}{2N}}dxdt=(m+\frac{1}{2N})\int_{0}^{\infty }\ell ^{m+%
\frac{1}{2N}-1}|\{U>\ell \}|d\ell \\
& ~~~=(m+\frac{1}{2N})\int_{0}^{\ell _{0}}\ell ^{m+\frac{1}{2N}-1}|\{U>\ell
\}|d\ell +(m+\frac{1}{2N})\int_{\ell _{0}}^{\infty }\ell ^{m+\frac{1}{2N}%
-1}|\{U>\ell \}|d\ell \\
& ~~~\leq c_{3}d^{N+2}\ell _{0}^{m+\frac{1}{2N}}+c_{4}\ell _{0}^{\frac{1}{2N}%
-\frac{2}{N}}(|\mu |(\Omega _{T}))^{\frac{2+N}{N}}.
\end{align*}%
Choosing $\ell _{0}=\left( \frac{|\mu |(\Omega _{T})}{d^{N}}\right) ^{\frac{%
N+2}{mN+2}}$, we get 
\begin{equation*}
\int_{Q}U^{(\lambda +1)(p-1)}dxdt\leq c_{5}d^{N+2}\left( \frac{|\mu |(\Omega
_{T})}{d^{N}}\right) ^{\frac{(N+2)(2mN+1)}{2mN(mN+2)}}.
\end{equation*}%
Thus, for a.e $(y,\tau )\in \Omega _{T},$ 
\begin{equation*}
U(y,\tau )\leq c_{6}\left( \left( \frac{|\mu |(\Omega _{T})}{d^{N}}\right)
^{m_{1}}+|\mu |(\Omega _{T})+1+\mathbb{I}_{2}^{2d}[|\mu |](y,\tau )\right)
\end{equation*}%
if $m>1$. Similarly, we also obtain for a.e $(y,\tau )\in \Omega _{T},$ 
\begin{equation*}
U(y,\tau )\leq c_{7}\left( \left( \frac{|\mu |(\Omega _{T})}{d^{N}}\right)
^{m_{2}}+1+\left( \mathbb{I}_{2}^{2d_{1}}[|\mu |](y,\tau )\right) ^{\frac{2}{%
2-N(1-m)}}\right) .
\end{equation*}%
if $m\leq 1$. By the comparison principle we get $|u|\leq U$ in $\Omega _{T}$%
, and \eqref{1106201419}-\eqref{1106201419b} follow.\medskip }
\end{proof}

\begin{lemma}
{\small \label{1106201420}Let $g\in C_{b}(\mathbb{R)}$ be nondecreasing with 
$g(0)=0$, and $\mu \in C_{b}(\Omega _{T})$. There exists a weak solution $%
u\in L^{\infty }(\Omega _{T})$ with $|u|^{m}\in L^{2}(0,T,H_{0}^{1}(\Omega
)) $ of problem 
\begin{equation}
\left\{ 
\begin{array}{l}
{u_{t}}-{\Delta }(|u|^{m-1}u)+g(u)=\mu ~~\text{in }\Omega _{T}, \\ 
{u}=0~~~\text{on }\partial \Omega \times (0,T), \\ 
u(0)=0~~~\text{ in }~\Omega .%
\end{array}%
\right.  \label{GO}
\end{equation}
}

{\small Moreover, the comparison principle holds for these solutions: if $%
u_{1},u_{2} $ are weak solutions of (\ref{GO}) when $(\mu ,g)$ is replaced
by $(\mu _{1},g_{1})$ and $(\mu _{2},g_{2})$, where $\mu _{1},\mu _{2}\in
C_b(\Omega _{T})$ with $\mu _{1}\geq \mu _{2}$ and $g_{1},g_{2}$ have the
same properties as $g$ with $g_{1}\leq g_{2}$ in $\mathbb{R}$ then $%
u_{1}\geq u_{2}$ in $\Omega _{T}$. }

{\small As a consequence, if $\mu \geq 0$ then $u\geq 0$. }
\end{lemma}

\begin{proof}[Proof of Lemma \protect\ref{1106201420}]
{\small Set $a_{n}(s)=m|s|^{m-1}$ if $1/n\leq |s|\leq n$ and $%
a_{n}(s)=m|n|^{m-1}$ if $|s|\geq n$, $a_{n}(s)=m(1/n)^{m-1}$ if $|s|\leq 1/n$%
. Also $A_{n}(\tau )=\int_{0}^{\tau }a_{n}(s)ds$. Then one can find $u_{n}$
being a weak solution to the following equation 
\begin{equation}
\left\{ 
\begin{array}{l}
{(u_{n})_{t}}-\text{div}(a_{n}(u_{n})\nabla u_{n})+g(u_{n})=\mu ~~\text{in }%
\Omega _{T}, \\ 
{u_{n}}=0~~~\text{on }\partial \Omega \times (0,T), \\ 
u_{n}(0)=0~~~\text{ in }~\Omega .%
\end{array}%
\right.  \label{ABC}
\end{equation}%
It is easy to see that $|u_{n}(x,t)|\leq t||\mu ||_{L^{\infty }(\Omega
_{T})} $ for all $(x,t)\in \Omega _{T}$. Thus, choosing $A_{n}(u_{n})$ as a
test function, we obtain 
\begin{equation}
\int_{\Omega _{T}}|\nabla A_{n}(u_{n})|^{2}dxdt\leq C_{1}(T,||\mu
||_{L^{\infty }(\Omega _{T})}).  \label{HIJ}
\end{equation}%
Now set $\Phi _{n}(\tau )=\int_{0}^{\tau }|A_{n}(s)|ds$. Choosing $%
|A_{n}(u_{n})|\varphi $ as a test function in \eqref{ABC}, where $\varphi
\in C_{c}^{2,1}(\Omega _{T})$, we get the relation in $~\mathcal{D}^{\prime
}(\Omega _{T}):$ 
\begin{equation*}
(\Phi _{n}(u_{n}))_{t}-\text{div}(|A_{n}(u_{n})|\nabla A_{n}(u_{n}))+\nabla
A_{n}(u_{n}).\nabla |A_{n}(u_{n})|+|A_{n}(u_{n})|g(u_{n})=|A_{n}(u_{n})|\mu .
\end{equation*}%
Hence, 
\begin{align*}
||(\Phi _{n}(u_{n}))_{t}||_{L^{1}(\Omega _{T})+L^{2}((0,T);H^{-1}(\Omega
))}& \leq ||A_{n}(u_{n})\nabla A_{n}(u_{n})||_{L^{2}(\Omega _{T})}+||\nabla
A_{n}(u_{n})|||_{L^{2}(\Omega _{T})}^{2} \\
& ~~+||A_{n}(u_{n})g(u_{n})||_{L^{1}(\Omega _{T})}+||A_{n}(u_{n})\mu
||_{L^{1}(\Omega _{T})}.
\end{align*}%
Combining this with (\ref{HIJ}) and the estimate $|A_{n}(u_{n})|\leq
C_{2}(T,||\mu ||_{L^{\infty }(\Omega )})$, we deduce that 
\begin{equation*}
\sup_{n}||(\Phi _{n}(u_{n}))_{t}||_{L^{1}(\Omega
_{T})+L^{2}(0,T,H^{-1}(\Omega ))}<\infty .
\end{equation*}%
On the other hand, since $|A_{n}(u_{n})|\leq |u_{n}|a_{n}(u_{n})\leq T||\mu
||_{L^{\infty }(\Omega )}a_{n}(u_{n}),$ there holds 
\begin{align*}
\int_{\Omega _{T}}|\nabla \Phi _{n}(u_{n})|^{2}dxdt& =\int_{\Omega
_{T}}|A_{n}(u_{n})|^{2}|\nabla u_{n}|^{2}dxdt\leq T||\mu ||_{L^{\infty
}(\Omega )}\int_{\Omega _{T}}|a_{n}(u_{n})|^{2}|\nabla u_{n}|^{2}dxdt \\
& \leq T||\mu ||_{L^{\infty }(\Omega )}\int_{\Omega _{T}}|\nabla
A_{n}(u_{n})|^{2}dxdt\leq C_{3}(T,||\mu ||_{L^{\infty }(\Omega )}).
\end{align*}%
Therefore, $\Phi _{n}(u_{n})$ is relatively compact in $L^{1}(\Omega _{T})$.
Note that 
\begin{equation*}
{\Phi _{n}}(s)=\left\{ 
\begin{array}{l}
\frac{m}{2}{\left( {\frac{1}{n}}\right) ^{m}}|s|^{2}\text{sign}(s)~~~\text{
if }~|s|\leq \frac{1}{n} \\ 
(m-1){\left( {\frac{1}{n}}\right) ^{m}}\left( {|s|-\frac{1}{n}}\right) \text{%
sign}(s)+\frac{1}{{m+1}}\left( {|s{|^{m+1}}-{{\left( {\frac{1}{n}}\right) }%
^{m+1}}}\right) \text{sign}(s)~~\text{ if }\frac{1}{n}\leq |s|\leq n. \\ 
\end{array}%
\right.
\end{equation*}%
So, for every $n_{1},n_{2}\geq n$ and $|s_{1}|,|s_{2}|\leq T||\mu
||_{L^{\infty }(\Omega )},$ 
\begin{equation*}
\frac{1}{m+1}||s_{1}|^{m}s_{1}-|s_{2}|^{m}s_{2}|\leq C_{4}(m,T||\mu
||_{L^{\infty }(\Omega )})\left( \frac{1}{n}\right) ^{m}+|{\Phi _{n_{1}}}%
(s_{1})-{\Phi _{n_{2}}}(s_{2})|.
\end{equation*}%
Hence, for any $\varepsilon >0,$ 
\begin{equation*}
\left\vert \left\{ \frac{1}{m+1}%
||u_{n_{1}}|^{m}u_{n_{1}}-|u_{n_{1}}|^{m}u_{n_{1}}|>2\varepsilon \right\}
\right\vert \leq |\left\{ |{\Phi _{n_{1}}}(u_{n_{1}})-{\Phi _{n_{2}}}%
(u_{n_{2}})|>\varepsilon \right\} |,
\end{equation*}%
for all $n_{1},n_{2}\geq \left( C_{4}(m,T||\mu ||_{L^{\infty }(\Omega
)})/\varepsilon \right) ^{1/m}$. Thus, up to a subsequence $\{u_{n}\}$
converges a.e in $\Omega _{T}$ to a function $u$. From (\ref{ABC}) we can
write 
\begin{equation*}
-\int_{\Omega _{T}}u_{n}\varphi _{t}dxdt-\int_{\Omega
_{T}}A_{n}(u_{n})\Delta \varphi dxdt+\int_{\Omega _{T}}g(u_{n})\varphi
dxdt=\int_{\Omega _{T}}\varphi d\mu ,
\end{equation*}%
for any $\varphi \in C_{c}^{2,1}(\Omega _{T})$. Thanks to the dominated
convergence Theorem we deduce that 
\begin{equation*}
-\int_{\Omega _{T}}u\varphi _{t}dxdt-\int_{\Omega _{T}}|u|^{m-1}u\Delta
\varphi dxdt+\int_{\Omega _{T}}g(u)\varphi dxdt=\int_{\Omega _{T}}\varphi
d\mu .
\end{equation*}%
By Fatou's lemma and \eqref{HIJ} we also get $|u|^{m}\in
L^{2}((0,T);H_{0}^{1}(\Omega ))$.\newline
Furthermore, by the classic maximum principle, see \cite[Theorem 9.7]{Li5},
if $\left\{ \tilde{u}_{n}\right\} $ is a sequence of solutions to equations %
\eqref{ABC} where $(g,\mu )$ is replaced by $(h,\nu )$ such that $\nu \in
C_{b}(\Omega _{T})$ with $\nu \geq \mu $ and $h$ has the same properties as $%
g$ satisfying $h\leq g$ in $\mathbb{R}$, then, $u_{n}\leq \tilde{u}_{n}$. As 
$n\rightarrow \infty $, we get $u\leq \tilde{u}$. This achieves the proof. }
\end{proof}

\begin{lemma}
{\small \label{1106201421}Let $m>\frac{N-2}{N}$ and $g:\mathbb{R}\rightarrow 
\mathbb{R}$ be a nondecreasing function, such that $g\in C_{b}( \mathbb{R})$%
, $g(0)=0$, and let $\mu \in \mathcal{M}_{b}(\Omega _{T})$. There exists a
very weak solution $u$ of equation \eqref{GO} which satisfies %
\eqref{1106201419}-\eqref{1106201419b} and 
\begin{equation}  \label{090720142}
\int_{\Omega _{T}}|g(u)|dxdt\leq |\mu |(\Omega _{T}), ~~
||u||_{L^{m+2/N,\infty }(\Omega _{T})}\leq C(|\mu |(\Omega _{T}))^{\frac{N+2%
}{mN+2}}.
\end{equation}
where $C=C(m,N)>0$. Moreover, the comparison principle holds for these
solutions: if $u_{1},u_{2} $ are very weak solutions of \eqref{GO} when $%
(\mu ,g)$ is replaced by $(\mu _{1},g_{1})$ and $(\mu _{2},g_{2})$, where $%
\mu _{1},\mu _{2}\in \mathcal{M}_{b}(\Omega _{T})$ with $\mu _{1}\geq \mu
_{2}$ and $g_{1},g_{2}$ have the same properties as $g$ with $g_{1}\leq
g_{2} $ in $\mathbb{R}$ then $u_{1}\geq u_{2}$ in $\Omega _{T}$. }
\end{lemma}

\begin{proof}
{\small Let $\left\{ \mu _{n}\right\} $ be a sequence in $C_{c}^{\infty
}(\Omega _{T})$ converging to $\mu $ in $\mathcal{M}_{b}(\Omega _{T}),$ such
that $|\mu _{n}|\leq \varphi _{n}\ast |\mu |$ and $|\mu _{n}|(\Omega
_{T})\leq |\mu |(\Omega _{T})$ for any $n\in \mathbb{N}$ where $\left\{
\varphi _{n}\right\} $ is a sequence of mollifiers in $\mathbb{R}^{N+1}$. By
Lemma \ref{1106201420} there exists a very weak solution $u_{n}$ of problem 
\begin{equation*}
\left\{ 
\begin{array}{l}
{(u_{n})_{t}}-{\Delta }(|u_{n}|^{m-1}u_{n})+g(u_{n})=\mu _{n}~~\text{in }%
\Omega _{T}, \\ 
{u_{n}}=0~~~\text{on }\partial \Omega \times (0,T), \\ 
u_{n}(0)=0~~~\text{ in }~\Omega ,%
\end{array}%
\right.
\end{equation*}%
which satisfies for a.e $(y,\tau )\in \Omega _{T}$, 
\begin{align*}
|u_{n}(y,\tau )|& \leq C\left( \left( \frac{|\mu |(\Omega _{T})}{d^{N}}%
\right) ^{m_{1}}+|\mu |(\Omega _{T})+1+\varphi _{n}\ast \mathbb{I}%
_{2}^{2d}[|\mu |](y,\tau )\right) ~\text{\quad if }m>1, \\
|u_{n}(y,\tau )|& \leq C\left( \left( \frac{|\mu |(\Omega _{T})}{d^{N}}%
\right) ^{m_{2}}+1+\left( \varphi _{n}\ast \mathbb{I}_{2}^{2d_{1}}[|\mu
|](y,\tau )\right) ^{\frac{2}{2-N(1-m)}}\right) ~\text{\quad if }m\leq 1,
\end{align*}%
and 
\begin{align}
& \int_{\Omega _{T}}|\nabla T_{k}(|u_{n}|^{m-1}u_{n})|^{2}dxdt\leq k|\mu
|(\Omega _{T}),\qquad \forall k>0,  \label{1106201424} \\
& |\{|u_{n}|>\ell \}|\leq C_{1}\ell ^{-\frac{2}{N}-m}|\mu |(\Omega _{T})^{%
\frac{N+2}{N}},\qquad \forall \ell >0,  \label{1106201425} \\
& \int_{\Omega _{T}}|g(u_{n})|dxdt\leq |\mu |(\Omega _{T}).  \notag
\end{align}%
For $l>0$, we consider $S_{l}\in C_{c}^{2}(\mathbb{R})$ such that 
\begin{equation*}
S_{l}(a)=|a|^{m}a,\quad \text{for }|a|\leq l,\quad \text{and\quad }%
S_{l}(a)=(2l)^{m+1}\text{sign}(a),\quad \text{for }|a|\geq 2l.
\end{equation*}%
Then we find the relation in $\mathcal{D}^{^{\prime }}(\Omega _{T}):$ 
\begin{equation*}
(S_{l}(u_{n}))_{t}-\text{div}\left( S_{l}^{^{\prime }}(u_{n})\nabla
(|u_{n}|^{m-1}u_{n})\right) +m|u_{n}|^{m-1}|\nabla u_{n}|^{2}S_{l}^{^{\prime
\prime }}(u_{n})+g(u_{n})S_{l}^{^{\prime }}(u_{n})=S_{l}^{^{\prime
}}(u_{n})\mu _{n}.
\end{equation*}%
It leads to 
\begin{align*}
||(S_{l}(u_{n}))_{t}||_{L^{1}(\Omega _{T})+L^{2}(0,T,H^{-1}(\Omega ))}& \leq
||S_{l}^{^{\prime }}(u_{n})\nabla (|u_{n}|^{m-1}u_{n})||_{L^{2}(\Omega
_{T})}+m|||u_{n}|^{m-1}|\nabla u_{n}|^{2}S_{l}^{^{\prime \prime
}}(u_{n})||_{L^{1}(\Omega _{T})} \\
& ~~+||g(u_{n})S_{l}^{^{\prime }}(u_{n})||_{L^{1}(\Omega
_{T})}+||S_{l}^{^{\prime }}(u_{n})\mu _{n}||_{L^{1}(\Omega _{T})}.
\end{align*}%
Since $|S_{l}^{^{\prime }}(u_{n})|\leq C_{2}\chi _{\lbrack -2l,2l]}(u_{n})$
and $|S_{l}^{^{\prime \prime }}(u_{n})|\leq C_{3}|u_{n}|^{m-1}\chi _{\lbrack
-2l,2l]}(u_{n})$, we obtain 
\begin{equation*}
||(S_{l}(u_{n}))_{t}||_{L^{1}(\Omega _{T})+L^{2}(0,T,H^{-1}(\Omega ))}\leq
C_{4}\left( ||\nabla T_{(2l)^{m}}(|u_{n}|^{m-1}u_{n})||_{L^{2}(\Omega
_{T})}+||g||_{L^{\infty }(\mathbb{R})}|\Omega _{T}|+|\mu _{n}|(\Omega
_{T})\right) .
\end{equation*}%
So from \eqref{1106201424} we deduce that $\left\{
(S_{l}(u_{n}))_{t}\right\} $ is bounded in $L^{1}(\Omega
_{T})+L^{2}((0,T);H^{-1}(\Omega ))$ and for any $n\in \mathbb{N},$ 
\begin{equation*}
||(S_{l}(u_{n}))_{t}||_{L^{1}(\Omega _{T})+L^{2}((0,T);H^{-1}(\Omega ))}\leq
C_{4}\left( (2l)^{m/2}(|\mu |(\Omega _{T}))^{1/2}+||g||_{L^{\infty }(\mathbb{%
R})}|\Omega _{T}|+|\mu |(\Omega _{T})\right) .
\end{equation*}%
Moreover, $\left\{ S_{l}(u_{n})\right\} $ is bounded in $%
L^{2}(0,T,H_{0}^{1}(\Omega ))$. Hence, $\left\{ S_{l}(u_{n})\right\} $ is
relatively compact in $L^{1}(\Omega _{T})$ for any $l>0$. Thanks to %
\eqref{1106201425} we find 
\begin{align*}
|\{||u_{n_{1}}|^{m}u_{n_{1}}-|u_{n_{1}}|^{m}u_{n_{1}}|>\ell \}|& \leq
|\{|u_{n_{1}}|>l\}|+|\{|u_{n_{2}}|>l\}|+|%
\{|S_{l}(u_{n_{1}})-S_{l}(u_{n_{2}})|>\ell \}| \\
& \leq 2C_{2}l^{-\frac{2}{N}-m}|\mu |(\Omega _{T})^{\frac{N+2}{N}%
}+|\{|S_{l}(u_{n_{1}})-S_{l}(u_{n_{2}})|>\ell \}|.
\end{align*}%
Thus, up to a subsequence $\{u_{n}\}$ converges a.e in $\Omega _{T}$ to a
function $u$. Consequently, $u$ is a very weak solution of equation %
\eqref{GO} and satisfies \eqref{090720142} and \eqref{1106201419}-%
\eqref{1106201419b}. The other conclusions follow in the same way. }
\end{proof}

\begin{remark}
{\small \label{090720141} If $\text{supp}(\mu )\subset \overline{\Omega }%
\times \lbrack a,T]$ for $a>0$, then the solution $u$ in Lemma \ref%
{1106201421} satisfies $u=0$ in $\Omega \times \lbrack 0,a)$. }
\end{remark}

{\small Now we recall the important property of Radon measures which was
proved in \cite{BaPi2} and \cite{H1}. }

\begin{proposition}
{\small \label{1006201410} Let $s>1$ and $\mu \in \mathcal{M}_{b}^{+}(\Omega
_{T})$. If $\mu $ is absolutely continuous with respect to $\text{Cap}%
_{2,1,s^{\prime }}$ in $\Omega _{T}$, there exists a nondecreasing sequence $%
\{\mu _{n}\}\subset \mathcal{M}_{b}^{+}(\Omega _{T})$, with compact support
in $\Omega_T$ which converges to $\mu $ weakly in $\mathcal{M}_{b}(\Omega
_{T})$ and satisfies $\mathbb{I}_{2}^{R}[\mu _{n}]\in L_{loc}^{s}(\mathbb{R}%
^{N+1})$ for all $R>0$. }
\end{proposition}

{\small Next we prove Theorem \ref{1106201412} in several steps of
approximation:\medskip\newline
}

\begin{proof}[Proof of Theorem \protect\ref{1106201412}]
{\small First suppose $m>1.$ Assume that $\mu ,\sigma $ are absolutely
continuous with respect to the capacities $\text{Cap}_{2,1,q^{\prime }}$ in $%
\Omega _{T} $ and $\text{Cap}_{\mathbf{G}_{\frac{2}{q}},q^{\prime }}$ in $%
\Omega $. Then $\sigma ^{+}\otimes \delta _{\{t=0\}}+\mu ^{+},\sigma
^{-}\otimes \delta _{\{t=0\}}+\mu ^{-}$ are absolutely continuous with
respect to the capacities $\text{Cap}_{2,1,q^{\prime }}$ in $\Omega \times
(-T,T)$. Applying Proposition \ref{1006201410} to $\sigma ^{+}\otimes \delta
_{\{t=0\}}+\mu ^{+},\sigma ^{-}\otimes \delta _{\{t=0\}}+\mu ^{-}$, there
exist two nondecreasing sequences $\{\upsilon _{1,n}\}$ and $\{\upsilon
_{2,n}\}$ of positive bounded measures with compact support in $\Omega
\times (-T,T)$ which converge respectively to $\sigma ^{+}\otimes \delta
_{\{t=0\}}+\mu ^{+}$ and $\sigma ^{-}\otimes \delta _{\{t=0\}}+\mu ^{-}$ in $%
\mathcal{M}_{b}(\Omega \times (-T,T))$ and such that $\mathbb{I}%
_{2}^{2d_{1}}[\upsilon _{1,n}],\mathbb{I}_{2}^{2d_{1}}[\upsilon _{2,n}]\in
L^{q}(\Omega \times (-T,T))$ for all $n\in \mathbb{N}$. By Lemma \ref%
{1106201421}, there exists a sequence $\{u_{n_{1},n_{2},k_{1},k_{2}}\}$ of
of weak solution of the problems 
\begin{equation*}
\left\{ 
\begin{array}{l}
{(u_{n_{1},n_{2},k_{1},k_{2}})_{t}}-{\Delta }%
(|u_{n_{1},n_{2},k_{1},k_{2}}|^{m-1}u_{n_{1},n_{2},k_{1},k_{2}})+T_{k_{1}}((u_{n_{1},n_{2},k_{1},k_{2}}^{+})^{q})
\\ 
~~~~~~~~~~~~-T_{k_{2}}((u_{n_{1},n_{2},k_{1},k_{2}}^{-})^{q})=\upsilon
_{1,n_{1}}-\upsilon _{2,n_{2}}~~\text{in }\Omega \times (-T,T), \\ 
{u_{n_{1},n_{2},k_{1},k_{2}}}=0~~~\text{on }\partial \Omega \times (-T,T),
\\ 
u_{n_{1},n_{2},k_{1},k_{2}}(-T)=0~~~\text{ in }~\Omega ,%
\end{array}%
\right.
\end{equation*}%
which satisfy 
\begin{align}
| u_{n_{1},n_{2},k_{1},k_{2}}|\leq C\left( \left(\frac{|\sigma |(\Omega
)+|\mu|(\Omega _{T})}{d^{N}}\right)^{m_1}+|\sigma|(\Omega )+|\mu|(\Omega
_{T})+1+\mathbb{I}_{2}^{2d}[\upsilon _{1,n_{1}}+\upsilon _{2,n_{2}}]\right) ,
\label{1106201428}
\end{align}%
and 
\begin{equation*}
\int_{\Omega
_{T}}T_{k_{1}}((u_{n_{1},n_{2},k_{1},k_{2}}^{+})^{q})dxdt+\int_{\Omega
_{T}}T_{k_{2}}((u_{n_{1},n_{2},k_{1},k_{2}}^{-})^{q})dxdt\leq |\mu |(\Omega
_{T}).
\end{equation*}%
Moreover, for any $n_{1}\in \mathbb{N},k_{2}>0$, $%
\{u_{n_{1},n_{2},k_{1},k_{2}}\}_{n_{2},k_{1}}$ is non-increasing and for any 
$n_{2}\in \mathbb{N},k_{1}>0$, $\{u_{n_{1},n_{2},k_{1},k_{2}}%
\}_{n_{1},k_{2}} $ is non-decreasing. Therefore, thanks to the fact that $%
\mathbb{I}_{2}^{2d_{1}}[\upsilon _{1,n}],\;\mathbb{I}_{2}^{2d_{1}}[\upsilon
_{2,n}]\in L^{q}(\Omega \times (-T,T))$ and from \eqref{1106201428} and the
dominated convergence Theorem, we deduce that $u_{n_{1},n_{2}}=\lim%
\limits_{k_{1}\rightarrow \infty }\lim\limits_{k_{2}\rightarrow \infty
}u_{n_{1},n_{2},k_{1},k_{2}}$ is a very weak solution of 
\begin{equation*}
\left\{ 
\begin{array}{l}
{(u_{n_{1},n_{2}})_{t}}-{\Delta }%
(|u_{n_{1},n_{2}}|^{m-1}u_{n_{1},n_{2}})+|u_{n_{1},n_{2}}|^{q-1}u_{n_{1},n_{2}}=\upsilon _{1,n_{1}}-\upsilon _{2,n_{2}}~~%
\text{in }\Omega \times (-T,T), \\ 
{u_{n_{1},n_{2}}}=0~~~\text{on }\partial \Omega \times (-T,T), \\ 
u_{n_{1},n_{2}}(-T)=0~~~\text{ in }~\Omega .%
\end{array}%
\right.
\end{equation*}%
And \eqref{1106201428} is true when $u_{n_{1},n_{2},k_{1},k_{2}}$ is
replaced by $u_{n_{1},n_{2}}$. Note that $\{u_{n_{1},n_{2}}\}_{n_{1}}$ is
non-increasing, $\{u_{n_{1},n_{2}}\}_{n_{2}}$ is non-decreasing and 
\begin{equation*}
\int_{\Omega _{T}}|u_{n_{1},n_{2}}|^{q}dxdt\leq |\mu |(\Omega _{T})~~\forall
~n_{1},n_{2}\in \mathbb{N}.
\end{equation*}%
From the monotone convergence Theorem we obtain that $u=\lim\limits_{n_{2}%
\rightarrow \infty }\lim\limits_{n_{1}\rightarrow \infty }u_{n_{1},n_{2}}$
is a very weak solution of 
\begin{equation*}
\left\{ 
\begin{array}{l}
{u_{t}}-{\Delta }(|u|^{m-1}u)+|u|^{q-1}u=\sigma \otimes \delta
_{\{t=0\}}+\chi _{\Omega _{T}}\mu ~~\text{in }\Omega \times (-T,T), \\ 
{u}=0~~~\text{on }\partial \Omega \times (-T,T), \\ 
u(-T)=0~~~\text{ in }~\Omega .%
\end{array}%
\right.
\end{equation*}%
which $u=0$ in $\Omega \times (-T,0)$ and $u$ satisfies \eqref{1106201430}.
Clearly, $u$ is a very weak solution of equation \eqref{one}. \newline
Next suppose $m\leq 1$. The proof is similar, with the new capacitary
assumptions and (\ref{1106201430}) is replaced by \eqref{1106201430b}%
.\medskip }
\end{proof}

{\small We also obtain the subcritical case. }

\begin{theorem}
{\small \label{subcri} Let $m>\frac{N-2}{N}$ and $0<q<m+\frac{2}{N}$. Then
problem \eqref{one} has a very weak solution for any $\mu \in \mathcal{M}%
_{b}(\Omega _{T})$ and $\sigma \in \mathcal{M}_{b}(\Omega ).$ }
\end{theorem}

\begin{proof}
{\small As the proof of Theorem \ref{1106201412}, we can reduce to the case $%
\sigma=0$. By Lemma \ref{1106201421}, there exists a very weak solution $%
u_{k_1,k_2}$ of 
\begin{equation*}
\left\{ 
\begin{array}{l}
{(u_{k_1,k_2})_{t}}-{\Delta }%
(|u_{k_1,k_2}|^{m-1}u_{k_1,k_2})+T_{k_1}((u_{k_1,k_2}^+)^q)-T_{k_2}((u_{k_1,k_2}^-)^q)=\mu ~~%
\text{in }\Omega _{T}, \\ 
{u_{n}}=0~~~\text{on }\partial \Omega \times (0,T), \\ 
u_{n}(0)=0~~~\text{ in }~\Omega.%
\end{array}%
\right.
\end{equation*}%
such that $\{u_{k_1,k_2}\}_{k_1}$ and $\{u_{k_1,k_2}\}_{k_2}$ are monotone
sequences and 
\begin{equation*}
||u_{k_1,k_2}||_{L^{m+2/N,\infty }(\Omega _{T})}\leq C(|\mu |(\Omega _{T}))^{%
\frac{N+2}{mN+2}}.
\end{equation*}
In particular, $\{u_{k_1,k_2}\}$ is a uniformly bounded in $L^s(\Omega_T)$
for any $0<s<m+\frac{2}{N}$.\newline
Therefore, we get that $u=\lim\limits_{k_{2}\rightarrow \infty
}\lim\limits_{k_{1}\rightarrow \infty }u_{k_{1},k_{2}}$ is a very weak
solution of \eqref{one}. This completes the proof. }
\end{proof}

{\small \medskip }

{\small Next, from an idea of \cite[Theorem 2.3]{BiH}, we obtain an
existence result \textit{for measures which present a good behaviour in time}%
: }

\begin{theorem}
{\small \label{1106201432} Let $m>\frac{N-2}{N},$ $q>\max (1,m)$ and $f\in
L^{1}(\Omega _{T})$, $\mu \in \mathcal{M}_{b}(\Omega _{T})$, such that 
\begin{equation*}
|\mu |\leq \omega \otimes F~~~\text{ for some }\omega \in \mathcal{M}%
_{b}^{+}(\Omega )\text{ and }F\in L^{1}_+((0,T)).
\end{equation*}%
If $\omega $ is absolutely continuous with respect to the capacity $\text{Cap%
}_{\mathbf{G}_{2},\frac{q}{q-m}}$ in $\Omega ,$ then there exists a very
weak solution to problem 
\begin{equation}
\left\{ 
\begin{array}{l}
{u_{t}}-{\Delta }(|u|^{m-1}u)+|u|^{q-1}u=f+\mu ~~\text{in }\Omega _{T}, \\ 
{u}=0~~~\text{on }\partial \Omega \times (0,T), \\ 
u(0)=0.%
\end{array}%
\right.  \label{0906201433}
\end{equation}
}
\end{theorem}

\begin{proof}
{\small For $R\in (0,\infty ]$, we define the $R$-truncated Riesz elliptic
potential of a measure $\nu \in \mathcal{M}_{b}^{+}(\Omega )$ by 
\begin{equation*}
\mathbf{I}_{2}^{R}[\nu ](x)=\int_{0}^{R}\frac{\nu (B_{\rho }(x))}{\rho ^{N-2}%
}\frac{d\rho }{\rho }~~~\forall x\in \Omega .
\end{equation*}%
By \cite[Theorem 2.6]{BiHV},there exists sequence $\{\omega _{n}\}\subset 
\mathcal{M}_{b}^{+}(\Omega )$ with compact support in $\Omega $ which
converges to $\omega $ in $\mathcal{M}_{b}(\Omega )$ and such that $\mathbf{I%
}_{2}^{2\text{diam}(\Omega )}[\omega _{n}]\in L^{q/m}(\Omega )$ for any $%
n\in \mathbb{N}$. We can write 
\begin{equation*}
f+\mu =\mu _{1}-\mu _{2},\qquad \mu _{1}=f^{+}+\mu ^{+},\qquad \mu
_{2}=f^{-}+\mu ^{-},
\end{equation*}%
and $\mu ^{+},\mu ^{-}\leqq \omega \otimes F.$ We set 
\begin{equation*}
\mu _{1,n}=T_{n}(f^{+})+\inf \{\mu ^{+},\omega _{n}\otimes T_{n}(F)\},\qquad
\mu _{2,n}=T_{n}(f^{-})+\inf \{\mu ^{-},\omega _{n}\otimes T_{n}(F)\}.
\end{equation*}%
Then $\left\{ \mu _{1,n}\right\} ,\left\{ \mu _{2,n}\right\} $ are
nondecreasing sequences converging to $\mu _{1},\mu _{2}$ respectively in $%
\mathcal{M}_{b}(\Omega _{T})$ and $\mu _{1,n},\mu _{2,n}\leq \tilde{\omega}%
_{n}\otimes \chi _{(0,T)},$ with $\tilde{\omega}_{n}=n(\chi _{\Omega
}+\omega _{n})$ and $\mathbf{I}_{2}^{2\text{diam}(\Omega )}[\tilde{\omega}%
_{n}]\in L^{q/m}(\Omega )$. As in the proof of Theorem \ref{1106201412},
there exists a sequence of weak solution $\{u_{n_{1},n_{2},k_{1},k_{2}}\}$
of equations 
\begin{equation}
\left\{ 
\begin{array}{l}
{(u_{n_{1},n_{2},k_{1},k_{2}})_{t}}-{\Delta }%
(|u_{n_{1},n_{2},k_{1},k_{2}}|^{m-1}u_{n_{1},n_{2},k_{1},k_{2}})+T_{k_{1}}((u_{n_{1},n_{2},k_{1},k_{2}}^{+})^{q})
\\ 
~~~~~~~~~~~~-T_{k_{2}}((u_{n_{1},n_{2},k_{1},k_{2}}^{-})^{q})=\mu
_{1,n_{1}}-\mu _{2,n_{2}}~\text{in }\Omega_T, \\ 
{u_{n_{1},n_{2},k_{1},k_{2}}}=0~~~\text{on }\partial \Omega \times (0,T), \\ 
u_{n_{1},n_{2},k_{1},k_{2}}(0)=0~~~\text{ in }~\Omega .%
\end{array}%
\right.
\end{equation}%
Using the comparison principle as in \cite{BiH}, we can assume that 
\begin{equation*}
-v_{n_{2}}\leq
|u_{n_{1},n_{2},k_{1},k_{2}}|^{m-1}u_{n_{1},n_{2},k_{1},k_{2}}\leq v_{n_{1}},
\end{equation*}%
where for any $n\in \mathbb{N},$ $v_{n}$ is a nonnegative weak solution of 
\begin{equation*}
\left\{ 
\begin{array}{l}
-{\Delta }v_{n}=\tilde{\omega}_{n}~~\text{in }\Omega , \\ 
{u_{n}}=0~~~\text{on }\partial \Omega ,%
\end{array}%
\right.
\end{equation*}%
such that 
\begin{equation*}
v_{n}\leq c_{1}\mathbf{I}_{2}^{2\text{diam}(\Omega )}[\tilde{\omega}%
_{n}]~~\forall ~n\in \mathbb{N}.
\end{equation*}%
Hence, utilizing the arguments in the proof of Theorem \ref{1106201412}, it
is easy to obtain the result as desired.\medskip }
\end{proof}

{\small It is easy to show that $\omega \otimes \chi _{\lbrack 0,T]}$ is
absolutely continuous with respect to the capacities $\text{Cap}_{2,1,\frac{q%
}{q-m},q^{\prime }}$ in $\Omega _{T}$ if any only if $\omega $ is absolutely
continuous with respect to the capacities $\text{Cap}_{\mathbf{G}_2,\frac{q}{%
q-m}}$ in $\Omega $. Consequently, we obtain the following: }

\begin{corollary}
{\small Let $m>\frac{N-2}{N},$ $q>\max (1,m)$ and $\omega \in \mathcal{M}%
_{b}(\Omega )$. Then, $\omega $ is absolutely continuous with respect to the
capacities $\text{Cap}_{\mathbf{G}_2,\frac{q}{q-m}}$ in $\Omega $ if and
only if there exists a very weak solution of problem 
\begin{equation}
\left\{ 
\begin{array}{l}
{u_{t}}-{\Delta }(|u|^{m-1}u)+|u|^{q-1}u=\omega \otimes \chi _{\lbrack
0,T]}~~\text{in }\Omega _{T}, \\ 
{u}=0~~~\text{on }\partial \Omega \times (0,T), \\ 
u(0)=0~~\text{ in }~\Omega .%
\end{array}%
\right.  \label{0906201434}
\end{equation}
}
\end{corollary}

\section{{\protect\small $p-$Laplacian evolution equation\label{PL}}}

{\small Here we consider solutions in the week sense of distributions, or in
the renormalized sense,. }

\subsection{\protect\small Distribution solutions}

\begin{definition}
{\small Let $\mu \in \mathcal{M}_{b}(\Omega _{T})$, $\sigma \in \mathcal{M}%
_{b}(\Omega )$ and $B\in C(\mathbb{R})$. A measurable function $u$ is a
distribution solution to problem \eqref{050420141} if $u\in
L^{s}(0,T,W_{0}^{1,s}(\Omega ))$ for any $s\in \left[ 1,p-\frac{N}{N+1}%
\right) ,$ and $B(u)\in L^{1}(\Omega _{T}),$ such that 
\begin{equation*}
-\int_{\Omega _{T}}u\varphi _{t}dxdt+\int_{\Omega _{T}}|\nabla
u|^{p-2}\nabla u.\nabla \varphi dxdt+\int_{\Omega _{T}}B(u)\varphi
dxdt=\int_{\Omega _{T}}\varphi d\mu +\int_{\Omega }\varphi (0)d\sigma,
\end{equation*}
for every $\varphi \in C_{c}^{1}(\Omega \times \lbrack 0,T))$. }
\end{definition}

\begin{remark}
{\small \label{110620143} Let $\sigma ^{\prime }\in \mathcal{M}_{b}(\Omega )$
and $a^{\prime }\in (0,T)$, set $\omega =\mu +\sigma ^{\prime }\otimes
\delta _{\{t=a^{\prime }\}}$. Let $u$ is a distribution solution to problem %
\eqref{050420141} with data $\omega $ and $\sigma =0,$ such that $\text{supp}%
(\mu )\subset \overline{\Omega }\times \lbrack a^{\prime },T]$, and $%
u=0,B(u)=0$ in $\Omega \times (0,a^{\prime })$. Then $\tilde{u}:=\left.
u\right\vert _{\Omega \times \lbrack a^{\prime },T)}$ is a distribution
solution to problem \eqref{050420141} in $\Omega \times (a^{\prime },T)$
with data $\mu $ and $\sigma ^{\prime }$. }
\end{remark}

\subsection{\protect\small Renormalized solutions}

{\small The notion of renormalized solution is stronger. It was first
introduced by Blanchard and Murat \cite{BlMu} to obtain uniqueness results
for the $p$-Laplace evolution problem for $L^{1}$ data $\mu $ and $\sigma $,
and developed by Petitta \cite{Pe08} for measure data $\mu $. It requires a
decomposition of the measure $\mu ,$ that we recall now.\medskip }

{\small Let $\mathcal{M}_{0}(\Omega _{T})$ be the space of Radon measures in 
$\Omega _{T}$ which are absolutely continuous with respect to the $C_{p}$%
-capacity, defined at (\ref{aaa}), and $\mathcal{M}_{s}(\Omega _{T})$ be the
space of measures in $\Omega _{T}$ with support on a set of zero $C_{p}$%
-capacity. Classically, any $\mu \in \mathcal{M}_{b}(\Omega _{T})$ can be
written in a unique way under the form $\mu =\mu _{0}+\mu _{s}$ where $\mu
_{0}\in \mathcal{M}_{0}(\Omega _{T})\cap \mathcal{M}_{b}(\Omega _{T})$ and $%
\mu _{s}\in \mathcal{M}_{s}(\Omega _{T})$. In turn $\mu _{0}$ can be
decomposed under the form 
\begin{equation*}
\mu _{0}=f-\text{div}~g+h_{t},
\end{equation*}%
where $f\in L^{1}(\Omega _{T})$, $g\in (L^{p^{\prime }}(\Omega _{T}))^{N}$
and $h\in L^{p}(0,T;W_{0}^{1,p}(\Omega )),$ see \cite{DrPoPr}; and we say
that $(f,g,h)$ is a decomposition of $\mu _{0}$. We say that a sequence of $%
\left\{ \mu _{n}\right\} $ in $\mathcal{M}_{b}(\Omega _{T})$ converges to $%
\mu \in \mathcal{M}_{b}(\Omega _{T})$ in the \textit{narrow topology} of
measures if 
\begin{equation*}
\lim_{n\rightarrow \infty }\int_{\Omega _{T}}\varphi d\mu _{n}=\int_{\Omega
_{T}}\varphi d\mu ~~\text{ }\forall \varphi \in C(\Omega _{T})\cap L^{\infty
}(\Omega _{T}).
\end{equation*}%
\medskip We recall that if $u$ is a measurable function defined and finite
a.e. in $\Omega _{T}$, such that $T_{k}(u)\in L^{p}(0,T,W_{0}^{1,p}(\Omega
)) $ for any $k>0$, there exists a measurable function $v:\Omega
_{T}\rightarrow \mathbb{R}^{N}$ such that $\nabla T_{k}(u)=\chi _{|u|\leq
k}v $ a.e. in $\Omega _{T}$ and for all $k>0$. We define the gradient $%
\nabla u$ of $u$ by $v=\nabla u$.\newline
}

\begin{definition}
{\small \label{defin}\label{100320142} Let $p>\frac{2N+1}{N+1}$ and $\mu
=\mu _{0}+\mu _{s}\in \mathcal{M}_{b}(\Omega _{T})$, $\sigma \in
L^{1}(\Omega )$ and $B\in C(\mathbb{R})$. A measurable function $u$ is a
renormalized solution of 
\begin{equation}
\left\{ 
\begin{array}{l}
{u_{t}}-\Delta _{p}u+B(u)=\mu ~\text{in }\Omega _{T}, \\ 
u=0~~~~~~~\text{on}~~\partial \Omega \times (0,T), \\ 
u(0)=\sigma ~~~\text{in}~~\Omega , \\ 
\end{array}%
\right.  \label{050420141}
\end{equation}%
if there exists a decomposition $(f,g,h)$ of $\mu _{0}$ such that 
\begin{align}
& v=u-h\in L^{s}((0,T);W_{0}^{1,s}(\Omega ))\cap L^{\infty
}((0,T);L^{1}(\Omega )),~\forall s\in \left[ 1,p-\frac{N}{N+1}\right) , 
\notag \\
& ~~~~T_{k}(v)\in L^{p}((0,T);W_{0}^{1,p}(\Omega ))~\forall k>0,B(u)\in
L^{1}(\Omega _{T}),  \label{defv}
\end{align}%
and:\medskip }

{\small (i) for any $S\in W^{2,\infty }(\mathbb{R})$ such that $S^{\prime }$
has compact support on $\mathbb{R}$, and $S(0)=0$, 
\begin{align}
& -\int_{\Omega }S(\sigma )\varphi (0)dx-\int_{\Omega _{T}}{{\varphi _{t}}%
S(v)}dxdt+\int_{\Omega _{T}}{S^{\prime }(v)|\nabla u|^{p-2}\nabla u\nabla
\varphi }dxdt  \notag \\
& ~~~~+\int_{\Omega _{T}}{S^{\prime \prime }(v)\varphi |\nabla
u|^{p-2}\nabla u\nabla v}dxdt+\int_{\Omega _{T}}{S^{\prime }(v)\varphi
B(u)dxdt= }\int_{\Omega _{T}}(f{S^{\prime }(v)\varphi }+g.\nabla ({S^{\prime
}(v)\varphi })dxdt  \label{100620143}
\end{align}%
for any $\varphi \in L^{p}((0,T);W_{0}^{1,p}(\Omega ))\cap L^{\infty
}(\Omega _{T})$ such that $\varphi _{t}\in L^{p^{\prime
}}((0,T);W^{-1,p^{\prime }}(\Omega ))+L^{1}(\Omega _{T})$ and $\varphi
(.,T)=0$;\medskip }

{\small (ii) for any $\phi \in C(\overline{\Omega _{T}}),$ 
\begin{equation}
\lim_{m\rightarrow \infty }\frac{1}{m}\int\limits_{\left\{ m\leq
v<2m\right\} }{\phi |\nabla u|^{p-2}\nabla u\nabla v}dxdt=\int_{\Omega
_{T}}\phi d\mu _{s}^{+}~~\text{ and }  \label{renor2}
\end{equation}%
\begin{equation}
\lim_{m\rightarrow \infty }\frac{1}{m}\int\limits_{\left\{ -m\geq
v>-2m\right\} }{\phi |\nabla u|^{p-2}\nabla u\nabla v}dxdt=\int_{\Omega
_{T}}\phi d\mu _{s}^{-}.  \label{renor3}
\end{equation}
}
\end{definition}

{\small We first mention a convergence result of \cite{BiH}. }

\begin{proposition}
{\small \label{100620148} Let $\{\mu _{n}\}$ be bounded in $\mathcal{M}%
_{b}(\Omega _{T})$ and $\left\{ \sigma _{n}\right\} $ be bounded in $%
L^{1}(\Omega ),$ and $B\equiv 0$. Let $u_{n}$ be a renormalized solution of %
\eqref{050420141} with data $\mu _{n}=\mu _{n,0}+\mu _{n,s}$ relative to a
decomposition $(f_{n},g_{n},h_{n})$ of $\mu _{n,0}$ and initial data $\sigma
_{n}$. If $\{f_{n}\}$ is bounded in $L^{1}(\Omega _{T})$, $\{g_{n}\}$
bounded in $(L^{p^{\prime }}(\Omega _{T}))^{N}$ and $\{h_{n}\}$ convergent
in $L^{p}(0,T,W_{0}^{1,p}(\Omega ))$, then, up to a subsequence, $\{u_{n}\}$
converges to a function $u$ in $L^{1}(\Omega _{T})$. Moreover, if $\{\mu
_{n}\}$ is bounded in $L^{1}(\Omega _{T})$ then $\{u_{n}\}$ is convergent in 
$L^{s}(0,T,W_{0}^{1,s}(\Omega ))$ for any $s\in \left[ 1,p-\frac{N}{N+1}%
\right) $. }
\end{proposition}

{\small Next we recall the fundamental stability result of \cite{BiH}. }

\begin{theorem}
{\small \label{100620145} Suppose that $p>\frac{2N+1}{N+1}$ and $B\equiv 0$.
Let $\sigma \in L^{1}(\Omega )$ and 
\begin{equation*}
\mu =f-\text{div} g+h_{t}+\mu _{s}^{+}-\mu _{s}^{-}\in \mathcal{M}_{b}({%
\Omega _{T}}),
\end{equation*}%
with $f\in L^{1}(\Omega _{T}),g\in (L^{p^{\prime }}(\Omega _{T}))^{N}$, $%
h\in L^{p}((0,T);W_{0}^{1,p}(\Omega ))$ and $\mu _{s}^{+},\mu _{s}^{-}\in 
\mathcal{M}_{s}^{+}(\Omega _{T})$. Let $\sigma _{n}\in L^{1}(\Omega )$ and 
\begin{equation*}
\mu _{n}=f_{n}-\text{div}g_{n}+(h_{n})_{t}+\rho _{n}-\eta _{n}\in \mathcal{M}%
_{b}(\Omega _{T}),
\end{equation*}%
with \ $f_{n}\in L^{1}(\Omega _{T}),g_{n}\in (L^{p^{\prime }}(\Omega
_{T}))^{N},h_{n}\in L^{p}((0,T);W_{0}^{1,p}(\Omega )),$ and $\rho _{n},\eta
_{n}\in \mathcal{M}_{b}^{+}(\Omega _{T}),$ such that 
\begin{equation*}
\rho _{n}=\rho _{n}^{1}-\text{div} ~\rho _{n}^{2}+\rho _{n,s},\qquad \eta
_{n}=\eta _{n}^{1}-\text{div} ~\eta _{n}^{2}+\eta _{n,s},
\end{equation*}%
with $\rho _{n}^{1},\eta _{n}^{1}\in L^{1}(\Omega _{T}),\rho _{n}^{2},\eta
_{n}^{2}\in (L^{p^{\prime }}(\Omega _{T}))^{N}$ and $\rho _{n,s},\eta
_{n,s}\in \mathcal{M}_{s}^{+}(\Omega _{T}).\medskip $ }

{\small Assume that $\{\mu _{n}\}$ is bounded in $\mathcal{M}_{b}(\Omega_T)$%
, $\{\sigma _{n}\},\{f_{n}\},\{g_{n}\},\{h_{n}\}$ converge to $\sigma ,f,g,h$
in $L^{1}(\Omega )$, weakly in $L^{1}(\Omega _{T})$, in $(L^{p^{\prime
}}(\Omega _{T}))^{N}$,in $L^{p}(0,T,W_{0}^{1,p}(\Omega ))$ respectively and $%
\{\rho _{n}\},\{\eta _{n}\}$ converge to $\mu _{s}^{+},\mu _{s}^{-}$ in the
narrow topology of measures; and $\left\{ \rho _{n}^{1}\right\}, \left\{
\eta _{n}^{1}\right\} $ are bounded in $L^{1}(\Omega _{T})$, and $\left\{
\rho _{n}^{2}\right\} ,\left\{ \eta _{n}^{2}\right\} $ bounded in $%
(L^{p^{\prime }}(\Omega _{T}))^{N}$. \newline
\medskip Let $\left\{ u_{n}\right\} $ be a sequence of renormalized
solutions of 
\begin{equation}
\left\{ 
\begin{array}{l}
{(u_{n})_{t}}-\Delta _{p}u_{n}=\mu _{n}~\text{in }\Omega _{T}, \\ 
u_{n}=0~~~~~~\text{ on }\partial \Omega \times (0,T), \\ 
u_{n}(0)=\sigma _{n}~\text{ in }\Omega , \\ 
\end{array}%
\right.  \label{100620146}
\end{equation}%
relative to the decomposition $(f_{n}+\rho _{n}^{1}-\eta _{n}^{1},g_{n}+\rho
_{n}^{2}-\eta _{n}^{2},h_{n})$ of $\mu _{n,0}.$ Let $v_{n}=u_{n}-h_{n}.$
\medskip }

{\small Then up to a subsequence, $\left\{ u_{n}\right\} $ converges $a.e.$
in $\Omega _{T}$ to a renormalized solution $u$ of \eqref{050420141}, and $%
\left\{ v_{n}\right\} $ converges $a.e.$ in $\Omega _{T}$ to $v=u-h.$
Moreover, $\left\{ \nabla v_{n}\right\} $ converge to $\nabla v$ a.e in $%
\Omega _{T},$ and $\left\{ T_{k}(v_{n})\right\} $ converges to $T_{k}(v) $
strongly in $L^{p}(0,T,W_{0}^{1,p}(\Omega ))$ for any $k>0 $. }
\end{theorem}

{\small In order to apply this Theorem, we need some the following
properties concerning approximate measures of $\mu \in \mathcal{M}%
_{b}^{+}(\Omega _{T})$, see also \cite{BiH}. }

\begin{proposition}
{\small \label{100620141}Let $\mu =\mu _{0}+\mu _{s}\in \mathcal{M}%
_{b}^{+}(\Omega _{T})$, $\mu _{0}\in \mathcal{M}_{0}(\Omega _{T})\cap 
\mathcal{M}_{b}^{+}(\Omega _{T})$ and $\mu _{s}\in \mathcal{M}_{s}(\Omega
_{T}).$ Let $\left\{ \varphi _{1,n}\right\} ,\left\{ \varphi _{2,n}\right\} $
be sequences of mollifiers in $\mathbb{R}^{N},\mathbb{R}$ respectively.
There exists a sequence of measures $\mu _{n,0}=(f_{n},g_{n},h_{n}),$ such
that $f_{n},g_{n},h_{n},\mu _{n,s}\in C_{c}^{\infty }(\Omega _{T})$ and
strongly converge to $f,g,h$ in $L^{1}(\Omega _{T}),(L^{p^{\prime }}(\Omega
_{T}))^{N} $ and $L^{p}((0,T);W_{0}^{1,p}(\Omega ))$ respectively, $\mu
_{n,s}$ converges to $\mu _{s}\in \mathcal{M}_{s}^{+}(\Omega _{T}),$ and $%
\mu _{n}=\mu _{n,0}+\mu _{n,s}$ converges to $\mu $, in the narrow topology,
and satisfying $0\leq \mu _{n}\leq (\varphi _{1,n}\varphi _{2,n})\ast \mu $,
and 
\begin{equation*}
||f_{n}||_{L^{1}(\Omega _{T})}+\left\Vert g_{n}\right\Vert _{(L^{p^{\prime
}}(\Omega _{T}))^{N}}+||h_{n}||_{L^{p}(0,T,W_{0}^{1,p}(\Omega ))}+\mu
_{n,s}(\Omega _{T})\leq 2\mu (\Omega _{T})~~\text{for any }~ n\in \mathbb{N}.
\end{equation*}
}
\end{proposition}

\begin{proposition}
{\small \label{100620144}Let $\mu =\mu _{0}+\mu _{s},\;\mu _{n}=\mu
_{n,0}+\mu _{n,s}\in \mathcal{M}_{b}^{+}(\Omega _{T})$ with $\mu _{0},\mu
_{n,0}\in \mathcal{M}_{0}(\Omega _{T})\cap \mathcal{M}_{b}^{+}(\Omega _{T})$
and $\mu _{n,s},\mu _{s}\in \mathcal{M}_{s}^{+}(\Omega _{T})$ such that $%
\left\{ \mu _{n}\right\} $ is nondecreasing and converges to $\mu $ in $%
\mathcal{M}_{b}(\Omega_T).$ Then, $\left\{ \mu _{n,s}\right\} $ is
nondecreasing and converging to $\mu _{s}$ in $\mathcal{M}_{b}(\Omega _{T});$
and there exist decompositions $(f,g,h)$ of $\mu _{0}$, $(f_{n},g_{n},h_{n})$
of $\mu _{n,0}$ such that $\left\{ f_{n}\right\} ,\left\{ g_{n}\right\}
,\left\{ h_{n}\right\} $ strongly converge to $f,g,h$ in $L^{1}(\Omega
_{T}),(L^{p^{\prime }}(\Omega _{T}))^{N}$ and $L^{p}((0,T);W_{0}^{1,p}(%
\Omega ))$ respectively, satisfying 
\begin{equation*}
||f_{n}||_{L^{1}(\Omega _{T})}+\left\Vert g_{n}\right\Vert _{(L^{p^{\prime
}}(\Omega _{T}))^{N}}+||h_{n}||_{L^{p}((0,T);W_{0}^{1,p}(\Omega ))}+\mu
_{n,s}(\Omega _{T})\leq 2\mu (\Omega _{T})~~\text{ for any } ~n\in\mathbb{N}.
\end{equation*}
}
\end{proposition}

\subsection{{\protect\small Proof of Theorem \protect\ref{100620149}}}

{\small Here the crucial point is a result of Liskevich, Skrypnik and Sobol 
\cite{LSS} for the $p$-Laplace evolution problem without absorption: }

\begin{theorem}
{\small \label{1106201414} Let $p>2,$ and $\mu \in \mathcal{M}_{b}(\Omega
_{T})$. If $u\in C([0,T];L_{loc}^{2}(\Omega ))\cap
L_{loc}^{p}(0,T,W_{loc}^{1,p}(\Omega ))$ is a distribution solution to
equation 
\begin{equation*}
{u_{t}}-{\Delta _{p}}u=\mu ~~\text{in }\Omega _{T},
\end{equation*}%
then there exists $C=C(N,p)$ such that, for every Lebesgue point $(x,t)\in
\Omega _{T}$ of $u$ and any $\rho >0$ such that $Q_{\rho ,\rho
^{p}}(x,t):=B_{\rho }(x)\times (t-\rho ^{p},t+\rho ^{p})\subset \Omega _{T}$
one has 
\begin{equation}
|u(x,t)|\leq C\left( 1+\left( \frac{1}{\rho ^{N+p}}\int_{Q_{\rho ,\rho
^{p}}(x,t)}|u|^{(\lambda +1)(p-1)}\right) ^{\frac{1}{1+\lambda (p-1)}}+%
\mathbf{P}_{p}^{\rho }[\mu ](x,t)\right) ,
\end{equation}%
where $\lambda =\min \{1/(p-1),1/N\}$ and 
\begin{equation*}
\mathbf{P}_{p}^{\rho }[\mu ](x,t)=\sum_{i=0}^{\infty }D_{p}(\rho _{i})(x,t),
\end{equation*}%
\begin{equation*}
D_{p}(\rho _{i})(x,t)=\inf_{\tau >0}\left\{ (p-2)\tau ^{-\frac{1}{p-2}}+%
\frac{1}{2(p-1)^{p-1}}\frac{|\mu |(Q_{\rho _{i},\tau \rho _{i}^{p}}(x,t))}{%
\rho _{i}^{N}}\right\} ,
\end{equation*}%
with $\rho _{i}=2^{-i}\rho $, $Q_{\rho ,\tau \rho ^{p}}(x,t)=B_{\rho
}(x)\times (t-\tau \rho ^{p},t+\tau \rho ^{p})$. }
\end{theorem}

{\small As a consequence, we deduce the following estimate: }

\begin{proposition}
{\small \label{100620147} If $u$ is a distribution solution of problem 
\begin{equation*}
\left\{ 
\begin{array}{l}
{u_{t}}-\Delta _{p}u=\mu ~\text{in }\Omega _{T}, \\ 
u=0~~~~~~\text{ on }\partial \Omega \times (0,T), \\ 
u(0)=0~\text{ in }\Omega , \\ 
\end{array}%
\right.
\end{equation*}%
with data $\mu \in C_{b}(\Omega _{T})$. Then there exists $C=C(N,p)$ such
that for $a.e.$ $(x,t)\in \Omega _{T},$ 
\begin{equation}
|u(x,t)|\leq C\left( 1+D+\left( \frac{|\mu |(\Omega _{T})}{D^{N}}\right)
^{m_{3}}+\mathbb{I}_{2}^{2D}[|\mu |](x,t)\right) ,  \label{12122013}
\end{equation}%
where $m_{3}$ and $D$ are defined at (\ref{lim}). }
\end{proposition}

\begin{proof}
{\small Let $x_{0}\in \Omega $ and $Q=B_{2D}(x_{0})\times
(-(2D)^{p},(2D)^{p}).$ Let $U\in C(Q)\cap
L^{p}((-(2D)^{p},(2D)^{p});W_{0}^{1,p}(B_{2D}(x_{0})))$ be the distribution
solution of 
\begin{equation}
\left\{ 
\begin{array}{l}
{U_{t}}-{\Delta _{p}}U=\chi _{\Omega _{T}}|\mu |\qquad \text{in }Q, \\ 
{u}=0~~~~~~~~~~~~~~~~\text{on }\partial B_{2D}(x_{0})\times
(-(2D)^{p},(2D)^{p}), \\ 
u(-(2D)^{p})=0~~~\text{ in }~B_{2D}(x_{0}),%
\end{array}%
\right.
\end{equation}%
where for $x_{0}\in \Omega $. Thus, by Theorem \ref{1106201414} we have, for
any $(x,t)\in \Omega _{T},$ 
\begin{equation}
U(x,t)\leq c_{1}\left( 1+\left( \frac{1}{D^{N+p}}%
\int_{Q_{D,D^{p}}(x,t)}|U|^{(\lambda +1)(p-1)}\right) ^{\frac{1}{1+\lambda
(p-1)}}+\mathbf{P}_{p}^{D}[\mu ](x,t)\right) ,  \label{090620145}
\end{equation}%
where $Q_{D,D^{p}}(x,t)=B_{D}(x)\times (t-D^{p},t+D^{p})$. \medskip \newline
According to Proposition 4.8 and Remark 4.9 of \cite{BiH}, there exists a
constant $C_{2}>0$ such that 
\begin{equation*}
|\{{|U|>\ell }\}|\leq c_{2}(|\mu |(\Omega _{T}))^{\frac{p+N}{N}}\ell ^{-p+1-%
\frac{p}{N}}~\qquad \forall \ell >0.
\end{equation*}%
Thus, for any $\ell _{0}>0,$ 
\begin{align*}
& \int_{Q}|U|^{(\lambda +1)(p-1)}dxdt=(\lambda +1)(p-1)\int_{0}^{\infty
}\ell ^{(\lambda +1)(p-1)-1}|\{|U|>\ell \}|d\ell \\
& ~~~=(\lambda +1)(p-1)\int_{0}^{\ell _{0}}\ell ^{(\lambda
+1)(p-1)-1}|\{|U|>\ell \}|d\ell +(\lambda +1)(p-1)\int_{\ell _{0}}^{\infty
}\ell ^{(\lambda +1)(p-1)-1}|\{|U|>\ell \}|d\ell \\
& ~~~\leq c_{3}D^{N+p}\ell _{0}^{(\lambda +1)(p-1)}+c_{4}\ell _{0}^{(\lambda
+1)(p-1)-p+1-\frac{p}{N}}(|\mu |(\Omega _{T}))^{\frac{p+N}{N}}.
\end{align*}%
Choosing $\ell _{0}=\left( \frac{|\mu |(\Omega _{T})}{D^{N}}\right) ^{\frac{%
N+p}{(p-1)N+p}}$, we get 
\begin{equation}
\int_{Q}|U|^{(\lambda +1)(p-1)}dxdt\leq c_{5}D^{N+p}\left( \frac{|\mu
|(\Omega _{T})}{D^{N}}\right) ^{\frac{(N+p)(\lambda +1)(p-1)}{(p-1)N+p}}.
\label{090620143}
\end{equation}%
Next we show that 
\begin{equation}
\mathbf{P}_{p}^{d_{2}}[\mu ](x,t)\leq (p-2)D+c_{6}\mathbb{I}_{2}^{2D}[|\mu
|](x,t).  \label{090620144}
\end{equation}%
Indeed, we have 
\begin{equation*}
D_{p}(\rho _{i})(x,t)\leq (p-2)\rho _{i}+\frac{1}{2(p-1)^{p-1}}\frac{|\mu |(%
\tilde{Q}_{\rho _{i}}(x,t))}{\rho _{i}^{N}},
\end{equation*}%
where $\rho _{i}=2^{-i}D$. Thus, 
\begin{align*}
\mathbf{P}_{p}^{D}[\mu ](x,t)& \leq (p-2)D+\frac{1}{2(p-1)^{p-1}}%
\sum_{i=0}^{\infty }\frac{|\mu |(\tilde{Q}_{\rho _{i}}(x,t))}{\rho _{i}^{N}}
\\
& \leq (p-2)D+C_{5}\int_{0}^{2D}\frac{|\mu |(\tilde{Q}_{\rho }(x,t))}{\rho
^{N}}\frac{d\rho }{\rho }.
\end{align*}%
So from \eqref{090620143}, \eqref{090620144} and \eqref{090620145} we get,
for any $(x,t)\in \Omega _{T},$ 
\begin{equation*}
|U(x,t)|\leq C\left( 1+D+\left( \frac{|\mu |(\Omega _{T})}{D^{N}}\right)
^{m_{3}}+\mathbb{I}_{2}^{2D}[|\mu |](x,t)\right) .
\end{equation*}%
By the comparison principle we get $|u|\leq U$ in $\Omega _{T}$, thus %
\eqref{12122013} follows.\medskip }
\end{proof}

\begin{proposition}
{\small \label{100620142} Let $p>2,$ and $\mu \in \mathcal{M}_{b}(\Omega
_{T}),$ $\sigma \in \mathcal{M}_{b}(\Omega ).$ There exists a distribution
solution $u $ of problem 
\begin{equation}
\left\{ 
\begin{array}{l}
{u_{t}}-{\Delta _{p}}u=\mu ~~\text{in }\Omega _{T}, \\ 
{u}=0~~~\text{on }\partial \Omega \times (0,T), \\ 
u(0)=\sigma .%
\end{array}%
\right.  \label{090620147}
\end{equation}%
which satisfies for any $(x,t)\in \Omega _{T}$ 
\begin{equation}
|u(x,t)|\leq C\left( 1+D+\left( \frac{|\sigma |(\Omega )+|\mu |(\Omega _{T})%
}{D^{N}}\right) ^{m_3}+\mathbb{I}_{2}^{2D}\left[ |\sigma |\otimes \delta
_{\{t=0\}}+|\mu |\right] (x,t)\right) ,  \label{090620146}
\end{equation}%
where $C=C(N,p)$. Moreover, if $\sigma \in L^{1}(\Omega )$, $u$ is a
renormalized solution. }
\end{proposition}

\begin{proof}
{\small Let $\{\varphi _{1,n}\},\{\varphi _{2,n}\}$ be sequences of standard
mollifiers in $\mathbb{R}^{N}$ and $\mathbb{R}$. Let $\mu =\mu _{0}+\mu
_{s}\in \mathcal{M}_{b}(\Omega _{T})$, with $\mu _{0}\in \mathcal{M}%
_{0}(\Omega _{T}),\mu _{s}\in \mathcal{M}_{s}(\Omega _{T})$. By Lemma \ref%
{100620141}, there exist sequences of nonnegative measures $\mu
_{n,0,i}=(f_{n,i},g_{n,i},h_{n,i})$ and $\mu _{n,s,i}$ such that $%
f_{n,i},g_{n,i},h_{n,i}\in C_{c}^{\infty }(\Omega _{T})$ and strongly
converge to some $f_{i},g_{i},h_{i}$ in $L^{1}(\Omega _{T}),(L^{p^{\prime
}}(\Omega _{T}))^{N}$ and $L^{p}((0,T);W_{0}^{1,p}(\Omega ))$ respectively,
and $\mu _{n,1},\mu _{n,2},\mu _{n,s,1},\mu _{n,s,2}\in C_{c}^{\infty
}(\Omega _{T})$ converge to $\mu ^{+},\mu ^{-},\mu _{s}^{+},\mu _{s}^{-}$ in
the narrow topology, with $\mu _{n,i}=\mu _{n,0,i}+\mu _{n,s,i}$, for $%
i=1,2, $ and satisfying 
\begin{equation*}
\mu _{0}^{+}=(f_{1},g_{1},h_{1}),\mu _{0}^{-}=(f_{2},g_{2},h_{2})~\text{ and 
}~0\leq \mu _{n,1}\leq (\varphi _{1,n}\varphi _{2,n})\ast \mu ^{+},0\leq \mu
_{n,2}\leq (\varphi _{1,n}\varphi _{2,n})\ast \mu ^{-}.
\end{equation*}%
Let $\sigma _{1,n},\sigma _{2,n}\in C_{c}^{\infty }(\Omega )$, converging to 
$\sigma ^{+}$ and $\sigma ^{-}$ in the narrow topology, and in $L^{1}(\Omega
)$ if $\sigma \in L^{1}(\Omega )$, such that 
\begin{equation*}
0\leq \sigma _{1,n}\leq \varphi _{1,n}\ast \sigma ^{+},0\leq \sigma
_{2,n}\leq \varphi _{1,n}\ast \sigma ^{-}.
\end{equation*}%
Set $\mu _{n}=\mu _{n,1}-\mu _{n,2}$ and $\sigma _{n}=\sigma _{1,n}-\sigma
_{2,n}$.\newline
Let $u_{n}$ be solution of the approximate problem 
\begin{equation}
\left\{ 
\begin{array}{l}
{(u_{n})_{t}}-\Delta _{p}u_{n}=\mu _{n}~~\text{in }\Omega _{T}, \\ 
u_{n}=0~~~\text{on}~\partial \Omega \times (0,T), \\ 
{u}_{n}(0)=\sigma _{n}~~\text{on }\Omega .%
\end{array}%
\right.
\end{equation}%
Let $g_{n,m}(x,t)=\sigma _{n}(x)\int_{-T}^{t}\varphi _{2,m}(s)ds$. As in
proof of Theorem 2.1 in \cite{H1}, by Theorem \ref{100620145}, there exists
a sequence $\{u_{n,m}\}_{m}$ of solutions of the problem 
\begin{equation}
\left\{ 
\begin{array}{l}
{(u_{n,m})_{t}}-\Delta _{p}u_{n,1,m}=\left( g_{n,m}\right) _{t}+\chi
_{\Omega _{T}}\mu _{n}~~\text{in }\Omega \times (-T,T), \\ 
u_{n,1,m}=0~~~\text{on}~\partial \Omega \times (-T,T), \\ 
{u}_{n,m}(-T)=0~~\text{on }\Omega ,%
\end{array}%
\right.
\end{equation}%
which converges to $u_{n}$ in $\Omega \times (0,T)$. By Proposition \ref%
{100620147}, there holds, for any $(x,t)\in \Omega _{T}$, 
\begin{equation*}
|u_{n,m}(x,t)|\leq C\left( 1+D+\left( \frac{|\mu _{n}|(\Omega _{T})+(|\sigma
_{n}|\otimes \varphi _{2,m})(\Omega \times (-T,T))}{D^{N}}\right) ^{m_{3}}+%
\mathbb{I}_{2}^{2D}[|\mu _{n}|+|\sigma _{n}|\otimes \varphi
_{2,m}](x,t)\right) .
\end{equation*}%
Therefore%
\begin{align*}
|u_{n,m}(x,t)|& \leq C\left( 1+D+\left( \frac{|\mu _{n}|(\Omega
_{T})+(|\sigma _{n}|\otimes \varphi _{2,m})(\Omega \times (-T,T))}{D^{N}}%
\right) ^{m_{3}}\right) \\
& ~~~~+C(\varphi _{1,n}\varphi _{2,m})\ast \mathbb{I}_{2}^{2D}[|\mu
|+|\sigma |\otimes \delta _{\{t=0\}}](x,t).
\end{align*}%
Letting $m\rightarrow \infty $, we get 
\begin{equation*}
|u_{n}(x,t)|\leq C\left( 1+D+\left( \frac{|\mu _{n}|(\Omega _{T})+|\sigma
_{n}|(\Omega )}{D^{N}}\right) ^{m_{3}}\right) +c_{1}(\varphi _{1,n})\ast (%
\mathbb{I}_{2}^{2D}[|\mu |+|\sigma |\otimes \delta _{\{t=0\}}](.,t))(x).
\end{equation*}%
Therefore, by Proposition \ref{100620148} and Theorem \ref{100620145} , up
to a subsequence, $\{u_{n}\}$ converges to a distribution solution $u$ of %
\eqref{090620147} (a renormalized solution if $\sigma \in L^{1}(\Omega )$),
and satisfying \eqref{090620146}. }
\end{proof}

\begin{proof}[Proof of Theorem \protect\ref{100620149}]
{\small \textbf{Step 1.} First, assume that $\sigma \in L^{1}(\Omega )$.
Because $\mu $ is absolutely continuous with respect to the capacity $\text{%
Cap}_{2,1,q^{\prime }}$, so are $\mu ^{+}$ and $\mu ^{-}$. Applying
Proposition \ref{1006201410} to $\mu ^{+},\mu ^{-}$, there exist two
nondecreasing sequences $\{\mu _{1,n}\}$ and $\{\mu _{2,n}\}$ of positive
bounded measures with compact support in $\Omega _{T}$ which converge to $%
\mu ^{+}$ and $\mu ^{-}$ in $\mathcal{M}_{b}(\Omega _{T})$ respectively and
such that $\mathbb{I}_{2}^{2D}[\mu _{1,n}],\mathbb{I}_{2}^{2D}[\mu
_{2,n}]\in L^{q}(\Omega _{T})$ for all $n\in \mathbb{N}$.\newline
For $i=1,2$, set $\tilde{\mu}_{i,1}=\mu _{i,1}$ and $\tilde{\mu}_{i,j}=\mu
_{i,j}-\mu _{i,j-1}\geq 0$, so $\mu _{i,n}=\sum_{j=1}^{n}\tilde{\mu}_{i,j}$.
We write 
\begin{equation*}
\mu _{i,n}=\mu _{i,n,0}+\mu _{i,n,s},\tilde{\mu}_{i,j}=\tilde{\mu}_{i,j,0}+%
\tilde{\mu}_{i,j,s},~\text{ with }\mu _{i,n,0},\tilde{\mu}_{i,n,0}\in 
\mathcal{M}_{0}(\Omega _{T}),\mu _{i,n,s},\tilde{\mu}_{i,n,s}\in \mathcal{M}%
_{s}(\Omega _{T}).
\end{equation*}%
Let $\{\varphi _{m}\}$ be a sequence of mollifiers in $\mathbb{R}^{N+1}$. As
in the proof of Proposition \ref{100620142}, for any $j\in \mathbb{N}$ and $%
i=1,2$, there exist sequences of nonnegative measures $\tilde{\mu}%
_{m,i,j,0}=(f_{m,i,j},g_{m,i,j},h_{m,i,j})$ and $\tilde{\mu}_{m,i,j,s}$ such
that $f_{m,i,j},g_{m,i,j},h_{m,i,j}\in C_{c}^{\infty }(\Omega _{T})$
strongly converge to some $f_{i,j},g_{i,j},h_{i,j}$ in $L^{1}(\Omega
_{T}),(L^{p^{\prime }}(\Omega _{T}))^{N}$ and $L^{p}(0,T,W_{0}^{1,p}(\Omega
))$ respectively; and $\tilde{\mu}_{m,i,j},\tilde{\mu}_{m,i,j,s}\in
C_{c}^{\infty }(\Omega _{T})$ converge to $\tilde{\mu}_{i,j},\tilde{\mu}%
_{i,j,s}$ in the narrow topology with $\tilde{\mu}_{m,i,j}=\tilde{\mu}%
_{m,i,j,0}+\tilde{\mu}_{m,i,j,s},$ which satisfy $\tilde{\mu}%
_{i,j,0}=(f_{i,j},g_{i,j},h_{i,j}),$ and 
\begin{equation*}
0\leq \tilde{\mu}_{m,i,j}\leq \varphi _{m}\ast \tilde{\mu}_{i,j},\tilde{\mu}%
_{m,i,j}(\Omega _{T})\leq \tilde{\mu}_{i,j}(\Omega _{T}),
\end{equation*}%
\begin{equation}
||f_{m,i,j}||_{L^{1}(\Omega _{T})}+\left\Vert g_{m,i,j}\right\Vert
_{(L^{p^{\prime }}(\Omega
_{T}))^{N}}+||h_{m,i,j}||_{L^{p}(0,T,W_{0}^{1,p}(\Omega ))}+\mu
_{m,i,j,s}(\Omega _{T})\leq 2\tilde{\mu}_{i,j}(\Omega _{T}).
\label{1006201411}
\end{equation}%
Note that, for any $n,m\in \mathbb{N},$ 
\begin{equation*}
\sum_{j=1}^{n}(\tilde{\mu}_{m,1,j}+\tilde{\mu}_{m,2,j})\leq \varphi _{m}\ast
(\mu _{1,n}+\mu _{2,n})\text{ and }\sum_{j=1}^{n}(\tilde{\mu}_{m,1,j}(\Omega
_{T})+\tilde{\mu}_{m,2,j}(\Omega _{T}))\leq |\mu |(\Omega _{T}).
\end{equation*}%
For any $n,k,m\in \mathbb{N}$, let $u_{n,k,m},v_{n,k,m}\in W$ be solutions
of problems 
\begin{equation}
\left\{ 
\begin{array}{l}
(u_{n,k,m})_{t}-\Delta
_{p}u_{n,k,m}+T_{k}(|u_{n,k,m}|^{q-1}u_{n,k,m})=\sum_{j=1}^{n}(\tilde{\mu}%
_{m,1,j}-\tilde{\mu}_{m,2,j})~\text{ in }~\Omega _{T}, \\ 
u_{n,k,m}=0~~~~~~~~~~~~~~~~\text{ on }\partial \Omega \times (0,T), \\ 
u_{n,k,m}(0)=T_{n}(\sigma ^{+})-T_{n}(\sigma ^{-})~~~\text{on }~\Omega ,%
\end{array}%
\right.  \label{1006201412}
\end{equation}%
and 
\begin{equation}
\left\{ 
\begin{array}{l}
(v_{n,k,m})_{t}-\Delta _{p}v_{n,k,m}+T_{k}(v_{n,k,m}^{q})=\sum_{j=1}^{n}(%
\tilde{\mu}_{m,1,j}+\tilde{\mu}_{m,2,j})~\text{ in }~\Omega _{T}, \\ 
v_{n,k,m}=0~~~~~~~~~~~~~~~~\text{ on }\partial \Omega \times (0,T), \\ 
v_{n,k,m}(0)=T_{n}(|\sigma |)~~~\text{on }~\Omega .%
\end{array}%
\right.  \label{1006201413}
\end{equation}%
By the comparison principle and Proposition \ref{1006201410} we have for any 
$m,k$ the sequences $\{v_{n,k,m}\}_{n}$ is increasing and 
\begin{align*}
|u_{n,k,m}|\leq v_{n,k,m}& \leq c_{1}\left( 1+D+\left( \frac{|\sigma
|(\Omega )+|\mu |(\Omega _{T})}{D^{N}}\right) ^{m_{3}}+\mathbb{I}_{2}^{2D}%
\left[ T_{n}(|\sigma |)\otimes \delta _{\{t=0\}}\right] \right) \\
& +c_{1}\varphi _{m}\ast \mathbb{I}_{2}^{2D}\left[ \mu _{1,n}+\mu _{2,n}%
\right] .
\end{align*}%
Moreover, 
\begin{equation*}
\int_{\Omega _{T}}T_{k}(v_{n,k,m}^{q})dxdt\leq |\mu |(\Omega _{T})+|\sigma
|(\Omega ).
\end{equation*}%
As in \cite[Proof of Lemma 6.4]{BiH}, thanks to Proposition \ref{100620148}
and Theorem \ref{100620145}, up to subsequences, $\{u_{n,k,m}\}_{m}$
converge to a renormalized solutions $u_{n,k}$ of problem 
\begin{equation*}
\left\{ 
\begin{array}{l}
(u_{n,k})_{t}-\Delta _{p}u_{n,k}+T_{k}(|u_{n,k}|^{q-1}u_{n,k})=\mu
_{1,n}-\mu _{2,n}\text{ in }~\Omega _{T}, \\ 
u_{n,k}=0~~~~~~~~~~~~~~~~\text{ on }\partial \Omega \times (0,T), \\ 
u_{n,k}(0)=T_{n}(\sigma ^{+})-T_{n}(\sigma ^{-})~~~\text{on }~\Omega ,%
\end{array}%
\right.
\end{equation*}%
relative to the decomposition $(\sum_{j=1}^{n}f_{1,j}-\sum_{j=1}^{n}f_{2,j},%
\sum_{j=1}^{n}g_{1,j}-\sum_{j=1}^{n}g_{2,j},\sum_{j=1}^{n}h_{1,j}-%
\sum_{j=1}^{n}h_{2,j})$ of $\mu _{1,n,0}-\mu _{2,n,0}$; and $%
\{v_{n,k,m}\}_{m}$ converge to a solution $v_{n,k}$ of 
\begin{equation*}
\left\{ 
\begin{array}{l}
(v_{n,k})_{t}-\Delta _{p}v_{n,k}+T_{k}(v_{n,k}^{q})=\mu _{1,n}+\mu _{2,n}~%
\text{ in }~\Omega _{T}, \\ 
v_{n,k}=0~~~~~~~~~~~~~~~~\text{ on }\partial \Omega \times (0,T), \\ 
v_{n,k}(0)=T_{n}(|\sigma |)~~~\text{on }~\Omega .%
\end{array}%
\right.
\end{equation*}%
relative to the decomposition $(\sum_{j=1}^{n}f_{1,j}+\sum_{j=1}^{n}f_{2,j},%
\sum_{j=1}^{n}g_{1,j}+\sum_{j=1}^{n}g_{2,j},\sum_{j=1}^{n}h_{1,j}+%
\sum_{j=1}^{n}h_{2,j})$ of $\mu _{1,n,0}+\mu _{2,n,0}$. And there holds 
\begin{equation*}
|u_{n,k}|\leq v_{n,k}\leq C\left( 1+D+\left( \frac{|\sigma |(\Omega )+|\mu
|(\Omega _{T})}{D^{N}}\right) ^{m_{3}}+\mathbb{I}_{2}^{2D}\left[
T_{n}(|\sigma |)\otimes \delta _{\{t=0\}}\right] \right) +C\mathbb{I}%
_{2}^{2D}\left[ \mu _{1,n}+\mu _{2,n}\right] .
\end{equation*}%
Observe that $\mathbb{I}_{2}^{2D}[\mu _{1,n}+\mu _{2,n}]\in L^{q}(\Omega
_{T})$ for any $n\in \mathbb{N}.$ Then, as in \cite[Proof of Lemma 6.5]{BiH}%
, thanks to Proposition \ref{100620148} and Theorem \ref{100620145}, up to a
subsequence, $\{u_{n,k}\}_{k}$ $\{v_{n,k}\}_{k}$ converge to renormalized
solutions $u_{n},v_{n}$ of problems 
\begin{equation}
\left\{ 
\begin{array}{l}
(u_{n})_{t}-\Delta _{p}u_{n}+|u_{n}|^{q-1}u_{n}=\mu _{1,n}-\mu _{2,n}~\text{
in }~\Omega _{T}, \\ 
u_{n}=0~~~~~~~~~~~~~~~~\text{ on }\partial \Omega \times (0,T), \\ 
u_{n}(0)=T_{n}(\sigma ^{+})-T_{n}(\sigma ^{-})~~~\text{in }~\Omega ,%
\end{array}%
\right.  \label{110620141}
\end{equation}%
\begin{equation}
\left\{ 
\begin{array}{l}
(v_{n})_{t}-\Delta _{p}v_{n}+v_{n}^{q}=\mu _{1,n}+\mu _{2,n}~\text{ in }%
~\Omega _{T}, \\ 
v_{n}=0~~~~~~~~~~~~~\text{ on }\partial \Omega \times (0,T), \\ 
v_{n}(0)=T_{n}(|\sigma |)~~~\text{in }~\Omega , \\ 
\end{array}%
\right.  \label{110620142}
\end{equation}%
which still satisfy 
\begin{equation*}
|u_{n}|\leq v_{n}\leq C\left( 1+D+\left( \frac{|\sigma |(\Omega )+|\mu
|(\Omega _{T})}{D^{N}}\right) ^{m_{3}}+\mathbb{I}_{2}^{2D}\left[
T_{n}(|\sigma |)\otimes \delta _{\{t=0\}}\right] \right) +C\mathbb{I}%
_{2}^{2D}\left[ \mu _{1,n}+\mu _{2,n}\right] .
\end{equation*}%
and the sequence $\{v_{n}\}_{n}$ is increasing and 
\begin{equation*}
\int_{\Omega _{T}}v_{n}^{q}dxdt\leq |\mu |(\Omega _{T})+|\sigma |(\Omega ).
\end{equation*}%
Note that from \eqref{1006201411} we have 
\begin{equation*}
||f_{i,j}||_{L^{1}(\Omega _{T})}+\left\Vert g_{i,j}\right\Vert
_{(L^{p^{\prime }}(\Omega
_{T}))^{N}}+||h_{i,j}||_{L^{p}(0,T,W_{0}^{1,p}(\Omega ))}\leq 2\tilde{\mu}%
_{i,j}(\Omega _{T}),
\end{equation*}%
which implies 
\begin{equation*}
||\sum_{j=1}^{n}f_{i,j}||_{L^{1}(\Omega
_{T})}+||\sum_{j=1}^{n}g_{i,j}||_{(L^{p^{\prime }}(\Omega
_{T}))^{N}}+||\sum_{j=1}^{n}h_{i,j}||_{L^{p}(0,T,W_{0}^{1,p}(\Omega ))}\leq 2%
\tilde{\mu}_{i,n}(\Omega _{T})\leq 2|\mu |(\Omega _{T}).
\end{equation*}%
Finally, as in \cite[Proof of Theorem 6.3]{BiH}, from Proposition \ref%
{100620148}, Theorem \ref{100620145} and the monotone convergence Theorem,
up to subsequences $\{u_{n}\}_{n}$, $\{v_{n}\}_{n}$ converge to a
renormalized solutions $u$, $v$ of problem 
\begin{equation*}
\left\{ 
\begin{array}{l}
u_{t}-\Delta _{p}u+|u|^{q-1}u=\mu ~\text{ in }~\Omega _{T}, \\ 
u=0~~~~~~~~~~~~~~~~\text{ on }\partial \Omega \times (0,T), \\ 
u(0)=\sigma ~~~\text{in }~\Omega ,%
\end{array}%
\right.
\end{equation*}%
relative to the decomposition $(\sum_{j=1}^{\infty
}f_{1,j}-\sum_{j=1}^{\infty }f_{2,j},\sum_{j=1}^{\infty
}g_{1,j}-\sum_{j=1}^{\infty }g_{2,j},\sum_{j=1}^{\infty
}h_{1,j}-\sum_{j=1}^{\infty }h_{2,j})$ of $\mu _{0}$. And 
\begin{equation*}
\left\{ 
\begin{array}{l}
v_{t}-\Delta _{p}v+v^q=\left\vert \mu \right\vert ~\text{ in }~\Omega _{T},
\\ 
v=0~~~~~~~~~~~~~~~~\text{ on }\partial \Omega \times (0,T), \\ 
v(0)=|\sigma |~~~\text{in }~\Omega ,%
\end{array}%
\right.
\end{equation*}%
relative to the decomposition $(\sum_{j=1}^{\infty
}f_{1,j}+\sum_{j=1}^{\infty }f_{2,j},\sum_{j=1}^{\infty
}g_{1,j}+\sum_{j=1}^{\infty }g_{2,j},\sum_{j=1}^{\infty
}h_{1,j}+\sum_{j=1}^{\infty }h_{2,j})$ of $|\mu _{0}|$ respectively; and 
\begin{equation*}
|u|\leq v\leq C\left( 1+D+\left( \frac{|\sigma |(\Omega )+|\mu |(\Omega _{T})%
}{D^{N}}\right) ^{m_{3}}+\mathbb{I}_{2}^{2D}\left[ |\sigma |\otimes \delta
_{\{t=0\}}+|\mu |\right] \right)
\end{equation*}%
Remark that, if $\sigma \equiv 0$ and $\text{supp}(\mu )\subset \overline{%
\Omega }\times \lbrack a,T]$, $a>0$, then $u=v=0$in $\Omega \times (0,a),$
since $u_{n,k}=v_{n,k}=0$ in $\Omega \times (0,a)$. \medskip \newline
\textbf{Step 2.} We consider any $\sigma \in \mathcal{M}_{b}(\Omega )$ such
that $\sigma $ is absolutely continuous with respect to the capacity $\text{%
Cap}_{\mathbf{G}_{\frac{2}{q}},q^{\prime }}$ in $\Omega $. So, $\mu +\sigma
\otimes \delta _{\{t=0\}}$ is absolutely continuous with respect to the
capacity $\text{Cap}_{2,1,q^{\prime }}$ in $\Omega \times (-T,T)$. As above,
we verify that there exists a renormalized solution $u$ of 
\begin{equation*}
\left\{ 
\begin{array}{l}
u_{t}-\Delta _{p}u+|u|^{q-1}u=\chi _{\Omega _{T}}\mu +\sigma \otimes \delta
_{\{t=0\}}~\text{ in }~\Omega \times (-T,T) \\ 
u=0~~~~~~~~~~\text{ on }\partial \Omega \times (-T,T), \\ 
u(-T)=0~~~\text{on }~\Omega ,%
\end{array}%
\right.
\end{equation*}%
satisfying $u=0$ in $\Omega \times (-T,0)$ and \eqref{080720141}. Finally,
from Remark \ref{110620143} we get the result. This completes the proof of
the Theorem. }
\end{proof}

\end{document}